\documentclass{article}

\usepackage{geometry}
\usepackage[utf8]{inputenc}
 \usepackage[margin=1cm]{caption}
\usepackage{wrapfig}
\usepackage{amssymb, amsmath, dsfont} 
\usepackage{float}
\usepackage{graphicx}
\usepackage{todonotes}
\usepackage{enumitem}
\usepackage{amsthm}
\usepackage[numbers]{natbib}
\usepackage{algpseudocode,algorithm,algorithmicx}
\usepackage[colorlinks,citecolor=blue]{hyperref}
\usepackage{cleveref}
\usepackage{multirow}
\usepackage{lscape}
\usepackage{comment}
\usepackage{hyperref}

\theoremstyle{plain}
\newtheorem{theorem}{Theorem}
\newtheorem{lemma}[theorem]{Lemma}

\newtheorem{proposition}[theorem]{Proposition}

\theoremstyle{definition}

\theoremstyle{remark}
\newtheorem{remark}[theorem]{Remark}
\theoremstyle{condition}
\newtheorem{condition}[theorem]{Condition}

\newcommand{\R}{\mathbb{R}}
\newcommand{\N}{\mathbb N}
\newcommand{\E}{\mathbb{E}}
\newcommand{\Var}{\mathbb{V}ar}
\newcommand{\var}{\Var}
\newcommand{\Cov}{\mathbb{C}ov}

\newcommand{\x}{\mathbf{x}}
\newcommand{\Prob}{\mathbb{P}}
\newcommand{\dist}{\mathcal{D}}

\newcommand{\LH}[1]{\left\|#1\right\|_{\mathcal{H}}}
\newcommand{\InerH}[2]{\left\langle #1,#2\right\rangle_{\mathcal{H}}}

\newcommand{\Ind}{\mathds{1}}
\newcommand{\ind}{\Ind}

\newcommand{\unitball}{\omega\in\mathcal{H}:\LH{\omega}=1}

\newcommand{\GCM}{\boldsymbol{\operatorname{GCM}}}
\newcommand{\wGCM}{\boldsymbol{\operatorname{wGCM}}}
\newcommand{\KGCM}{\boldsymbol{\operatorname{KGCM}}}
\newcommand{\MMD}{\boldsymbol{\operatorname{MMD}}}
\newcommand{\Q}{Q}

\newtheorem{innercustomthm}{Condition}
\newenvironment{customCondition}[1]
  {\renewcommand\theinnercustomthm{#1}\innercustomthm}
  {\endinnercustomthm}
  
  \crefname{innercustomthm}{Condition}{Conditions}

\title{A general framework for the analysis of kernel-based tests}
\author{Tamara Fern\'andez \thanks{Faculty of Engineering and Science, Universidad Adolfo Ib\'a\~nez, Chile. t.a.fernandez.aguilar@gmail.com} \and Nicol\'as Rivera\thanks{Instituto de ingenier\'ia Matem\'atica, Universidad de Valpara\'iso, Chile. n.a.rivera.aburto@gmail.com}}
\date{\today}

\begin{document}

\maketitle
\begin{abstract}
{Kernel-based tests provide a simple yet effective framework that use the theory of reproducing kernel Hilbert spaces to design non-parametric testing procedures. In this paper we propose new theoretical tools that can be used to study the asymptotic behaviour of kernel-based tests in several data scenarios, and in many different testing problems. Unlike current approaches, our methods avoid using lengthy $U$ and $V$ statistics expansions and limit theorems, that commonly appear in the literature, and works directly with random functionals on Hilbert spaces. Therefore, our framework leads to a much simpler and clean analysis of kernel tests, only requiring  mild regularity conditions. Furthermore, we show that, in general, our analysis cannot be improved by proving that the regularity conditions required by our methods are both sufficient and necessary. To illustrate the effectiveness of our approach we present a new kernel-test for the conditional independence testing problem, as well as new analyses for already known kernel-based tests.}
\end{abstract}
\section{Introduction}\label{sec:introduction}

The aim of this paper is to introduce new general tools for the analysis of kernel-based methods in the context of hypothesis testing. For this purpose,  consider the following general framework. Let  $X_1,\ldots, X_n$ be a collection of data points, and consider a null hypothesis of interest which we denote by $H_0$. Suppose that to assess the validity of $H_0$ we have access to a test-statistic $S_n(\omega)$ that depends (implicitly) on our data and on a deterministic real-valued weight function $\omega:\mathcal X\to \R$. Additionally assume that $S_n(\omega)$ satisfies that  $S_n(\omega)\approx 0$ for any fixed $\omega$ under the null hypothesis. A test procedure based on $S_n(\omega)$ works as follows. Choose a function $\omega$ and compute $S_n(\omega)$. If we observe  $S_n(\omega)\approx 0$ then we do not reject the null hypothesis, but if we observe that $S_n(\omega)$ is far away from zero, then we use this as evidence to reject the null. Of course, we might still have that  $S_n(\omega)\approx0$ under the alternative hypothesis for some functions $\omega$, and in such case the test-statistic $S_n(\omega)$ will perform poorly.

In the previous setting, we need to carefully choose the weight function $\omega$ in order to have a robust test that is able to distinguish between the null and the alternative hypotheses. Arguably, there are two main approaches that can be followed to tackle the problem of choosing an appropriate weight function: we can learn the weight function $\omega$ from our data, i.e. an adaptive weight approach, or we can combine several weight functions into a single test-statistic, which is the path we follow in this paper. 

Following the idea of combining several weight functions into a single test, a simple proposal is to consider as test-statistic the random variable
\begin{align}\label{eqn:defiTn}
    \Psi_n=\sup_{\omega\in\mathcal{H}:\LH{\omega}=1}S_n(\omega)^2,
\end{align}
where $\mathcal{H}$ is a separable reproducing kernel Hilbert space (RKHS). While in principle we could have taken supremum over an arbitrary space of functions, we chose the unit ball of an RKHSs since the nice properties of RKHSs in addition to some regularity conditions over the general test-statistic $S_n(\omega)$ will guarantee good properties of $\Psi_n$, which will be fundamental to construct a testing procedure. In particular, we will assume that $S_n(\omega)$ is linear on its argument $\omega$, that is  $S_n(a\omega_1+b\omega_2) = aS_n(\omega_1)+bS_n(\omega_2)$ for any $a,b\in \R$ and $\omega_1,\omega_2\in\mathcal{H}$, then it will follow that $\Psi_n$ can be evaluated exactly via a closed-form expression. In the previous case we say that $S_n(\omega)$ is a linear test-statistic, and we refer to $\Psi_n$ to as the kernelisation of $S_n$. This simple idea is the base of what is known as a kernel-based test \cite{gretton2006kernel}, and has been implicitly and explicitly applied in many contexts.  In the literature, testing procedures based on test-statistics of the form of \cref{eqn:defiTn} are usually refer to as kernel test-statistics, and the whole testing procedure derived from it to as a kernel test. 

To illustrate the kernel procedure based on a test-statistic of the form \cref{eqn:defiTn}, we consider the best known example of kernel tests: the maximum mean discrepancy ($\MMD$), which was introduced in the seminal work of  \citet{gretton2006kernel}. The $\MMD$ is used to assess the null hypothesis $H_0:F_0=F_1$, where $F_0$ and $F_1$ are both unknown distribution functions, based on two independent samples $X_1,\ldots,X_{n_0}\overset{i.i.d.}{\sim}F_0$ and $Y_1,\ldots,Y_{n_1}\overset{i.i.d.}{\sim}F_1$. Let $n = n_0+n_1$, then the $\MMD$ is defined as
\begin{align*}
\MMD(\{X_i\}_{i=1}^{n_0},\{Y_i\}_{i=1}^{n_1})=\sup_{\unitball}S_n(\omega)\qquad\text{where}\qquad S_n(\omega)= \frac{1}{n_0}\sum_{i=1}^{n_0} \omega(X_i)-\frac{1}{n_1}\sum_{j=1}^{n_1}\omega(Y_j).
\end{align*}

Notice that the $\MMD^2(\{X_i\}_{i=1}^{n_0},\{Y_i\}_{i=1}^{n_1})$ is clearly of the form of \cref{eqn:defiTn}, and that $S_n$ satisfies the properties previously described, that is, it is linear on $\omega$. Moreover, under the null hypothesis (and for large $n$) we expect $S_n(\omega)\approx0$. However, observe that $S_n(\omega)$ might give values close to 0 under the alternative if we choose a bad weight function (e.g. consider the trivial example where $\omega = 0$). Thus, taking supremum over a space of functions is a sensible choice to make $S_n$ more robust against alternatives.

Since the development of the $\MMD$, a lot of research has been conducted in the area of kernel-based test.  In the context of Goodness-of-Fit the most common examples have been proposed by kernelising a Stein's operator, resulting in tests that are commonly referred to as Kernel Stein Discrepancy (KSD) tests, and have been obtained for several data-domains such as $\R^d$ \cite{chwialkowski2016kernel,liu2016kernelized}, point processes \cite{yang2019stein}, random graph models \cite{xu2021stein}, among other. Kernel methods have also been applied for the problem of testing independence, where the main objective has been the analysis of the so-called Hilbert-Schmidt independence  criterion (HSIC) \cite{gretton2005measuring,gretton2007kernel,smola2007hilbert}. Other testing problems where kernel methods have been applied are: conditional independence \cite{zhang2012kernel, doran2014permutation}, composite goodness-of-fit \cite{key2021composite}, and several testing problems in survival analysis \cite{fernandez2020kernelized,fernandez2021reproducing,fernandez2021kernel, ditzhaus2022multiple}. One of the advantages of kernel methods is that they can virtually be applied to any type of data structure, including graphs, strings, sets, etc, and so testing procedures for those type of data can be designed. We refer to reader to \cite{chen2020kernel} for an introductory review of kernel-based tests, and to \cite{muandet2017kernel, hofmann2008kernel} for a general overview of kernel methods and their use in Statistics and Machine Learning.

Despite the vast literature on kernel-based methods, up to the best of our knowledge there are no works aiming towards finding a unified framework to analyse kernel-based tests, and thus, most of the existing tests and results are derived and analysed by using first principles in a case-by-case basis despite the fact that many similarities are presented in the analyses: i) most previous works have as a main object of interest a random variable of the form of $\Psi_n$ which arises, implicitly or explicitly, from a linear test-statistic $S_n(\omega)$, ii) most works base their analyses on writing $\Psi_n$ in a $V$-statistic shape, iii) the limit distribution of $\Psi_n$ can be expressed as a (infinite) linear combination of independent chi-squared random variables, and it is derived from $V$-statistics limit theorems, iv) in most works the same re-sample schemes are used, being wild bootstrap the most common.

In an effort to organise the common ideas found in the literature, we provide a unified analysis of kernelised test-statistics. Our main goals are: i) to avoid lengthy computations that usually appear by expressing the kernelised test-statistic $\Psi_n$ as an object that resembles a $U/V$-statistic, ii) to be able to use already known results for $S_n(\omega)$ in the analysis of $\Psi_n$ (which are usually much easier to obtain), and iii) to provide a systematic approach to analyse the asymptotic behaviour of $\Psi_n$ and resampling schemes.

Our main idea to achieve our goals is to completely avoid expanding $\Psi_n$ as a $U/V$-statistic and work directly by looking at $S_n$ as a random functional on the Hilbert space $\mathcal H$, and looking for conditions that allow us to extrapolate limiting results of $S_n(\omega)$, for fixed $\omega \in \mathcal H$, to $\Psi_n$ . Working with random functionals is much simpler, and indeed our analysis is based on first principles of Hilbert space valued random variables. At a high level, we have that 
\begin{itemize}
    \item[i)] Under the null hypothesis, some regularity conditions, and appropriate scaling it holds that
    \begin{align*}
\text{if for all $\omega\in\mathcal{H}$,}\quad S_n(\omega)\overset{\dist}{\to} N(0,\sigma^2_{\omega})\quad \text{then} \quad \Psi_n\overset{\dist}{\to} \sum_{i=1}^\infty\lambda_iZ_i^2,
\end{align*}
when the number of data points $n$ tends to infinity. The variables $Z_i$ are i.i.d. standard normal random variables, and $\lambda_1,\lambda_2,\ldots$ are non-negative constants. More details are given in \Cref{thm:ConvergenceSumChiSq}.  
\item[ii)] Under the alternative hypothesis, some regularity conditions, and appropriate scaling it holds that
\begin{align*}
   \text{if for all $\omega\in\mathcal{H}$,}\quad  S_n(\omega){\to} c(\omega)\quad\text{then}\quad \Psi_n{\to} ({c}^\star)^2 = \sup_{\unitball} c(\omega)^2,
\end{align*}
where convergence holds almost surely or in probability. More details are given in \Cref{Thm:alternati}.
\end{itemize}

We will see that our approach not only can be applied to the kernelised test-statistic $\Psi_n$ but also to a bootstrapped version of it. Hence, our methods not only gives us limiting results about $\Psi_n$, but gives asymptotic guarantees for the whole testing procedure based on the test-statistic $\Psi_n$.

To show the effectiveness of our approach, we provide three applications in \Cref{sec:examples}. The first two are already known in the literature, so we provide a new yet very concise and much simpler analysis with our tools. The third application is a novel test for conditional independent testing that kernelises the very recently proposed Weighted Generalised Covariance Measure \cite{ scheidegger2021weighted} which is a weighted generalisation of the introduced Generalised Covariance Measure \cite{shah2020hardness}.

For the rest of the paper we adopt standard notation used in Statistics. In particular, we write $\overset{a.s}{\to}$, $\overset{\Prob}{\to}$, and $\overset{\mathcal D}{\to}$ to denote convergence almost surely, in probability and in distribution (in law), respectively. All limits are taken when $n$, the number of data points, tend to infinity unless explicitly said otherwise.

\section{Convergence of kernelised linear test-statistics}\label{sec:convergence}

Consider a Hilbert space $\mathcal H$ of real functions $\omega: \mathcal X \to \R$ with inner product $\InerH{\cdot}{\cdot}$. We say that $\mathcal H$ is a reproducing kernel Hilbert space (RKHS) if the evaluation functional $E_x:\omega\to \omega(x)\in \R$ is bounded for each $x\in \mathcal X$. Then by the Riesz's representation theorem it exists a unique $K_x\in \mathcal H$ such that $E_x\omega = \InerH{K_x}{\omega}$. For $x,y\in \mathcal X$, we denote by $K(x,y) = \InerH{K_x}{K_y}$ the so-called reproducing kernel of $\mathcal H$ which is a symmetric and positive definite function $\mathcal X \times \mathcal X\to \R$. The kernel function $K$ characterises $\mathcal H$, and indeed, given a symmetric positive definite function $K$ there is a unique RKHS with such a function as reproducing kernel. In practice, we do not choose $\mathcal H$, but rather the kernel $K$. Standard kernel functions are the squared-exponential, the Ornstein–Uhlenbeck, and the rational quadratic kernels, see \cite[Chapter 2]{duvenaud2014automatic} for a short compendium of kernel functions.

Since we are interested on random variables on the RKHS $\mathcal H$ and on the space of bounded linear functionals $\mathcal H^\star$ (with the standard operator norm  $\| \cdot \|_{\mathcal H \to \R}$), we will assume that all the RKHSs in this work are separable (and so is the space of linear functionals). A sufficient condition to ensure that a RKHS of functions $\mathcal X\to \R$ is separable is that $\mathcal X$ is separable, and that the kernel function $K$ is continuous on $\mathcal X\times \mathcal X$. We assume some underlying probability space $(\Omega,\mathcal F, \Prob)$ and we consider random variables $\Omega\to \mathcal H$ or $\Omega \to \mathcal H^\star$ that are measurable with respect to the corresponding Borel-sigma algebra (other natural forms of measurability, such as the cylindrical sigma-algebra, are equivalent in the setting of separable Hilbert spaces). Since there is an isometry between $\mathcal H$ and $\mathcal H^\star$ we can define random variables in one space and move to its representation on the other space without worrying about measurability issues. In particular, given a random bounded functional $S$, the random variable $\|S\|_{\mathcal H \to \R} = \sup_{\unitball}S(\omega)$ is measurable, and for a random variable $\xi \in \mathcal H$, its norm $\LH{\xi}$ is also measurable. 

 In order  to keep the statistical motivation in our presentation, from now on, we refer to a random variable on $\mathcal H^\star$ to as a bounded linear test-statistic. We are interested on sequences of $(S_n)_{n\in \N}$ of bounded linear test-statistics in $\mathcal H^\star$, and more specifically on the random variable
\begin{align}
    \Psi_n:= \|S_n\|_{\mathcal H \to \R}^2= \sup_{\omega\in \mathcal H: \LH{\omega}=1} S_n(\omega)^2.
\end{align}

As discussed in the introduction, in practical scenarios of hypothesis testing, $S_n(\omega)$ is some simple test-statistic depending on the weight $\omega$, and the randomness is usually provided by the observed data (usually $n$ data points $X_1,\ldots, X_n$, making sense of the subindex $n$ as well). There are two cases of interest here. The first case is when $S_n(\omega)$ is already properly scaled and it converges to a normal distribution with mean 0 and variance $\sigma_{\omega}^2$ for each function $\omega$. This is the typical scenario under the null hypothesis. The other interesting case is when $S_n(\omega)$ converges (under proper scaling) to a constant $c(\omega)$ that depends on $\omega$ (and hopefully $c(\omega)\neq 0$), which usually holds under the alternative hypothesis. 

Let's start analysing the first case as it is the most interesting one. We start by assuming that the sequence of linear statistics $(S_n)_{n\geq1}$ satisfies the following \Cref{Cond:0bilinear}, where the G stands for ``Gaussian" as in Gaussian distribution.

\begin{customCondition}{$G_0$}\label{Cond:0bilinear}
There exists a continuous bilinear form $\sigma:\mathcal{H}\times\mathcal{H}\to\R$ such that for any $m\in\N$, the bounded linear test-statistic $S_n(w_1+\ldots +w_m)$ converges in distribution to a normal random variable with mean 0 and variance given by $\sum_{i=1}^m\sum_{j=1}^m \sigma(w_i,w_j)$. 
\end{customCondition}

\Cref{Cond:0bilinear} is rather natural and it is the most common behaviour for test-statistics under the null hypothesis, so more than a condition it is a framework.  

The bilinear form $\sigma:\mathcal{H}\times\mathcal{H}\to\R$ of \Cref{Cond:0bilinear} plays a fundamental role in our analysis as it will be important to define the potential limit (in law) of $S_n$. In this context, we are  interested on the linear transformation $T_{\sigma}:\mathcal H \to \mathcal H$ given by
 \begin{align}\label{eqn:Tsigma}
     (T_{\sigma}\omega)(x)=\sigma(\omega, K_x),\quad \forall x\in\mathcal{X}.
 \end{align}
 Recall that $K_x \in \mathcal H$ is the (unique) element associated with the evaluation function $E_x$ via the Riesz representation theorem. Assuming \Cref{Cond:2Tail} below we can show that $T_{\sigma}$  has co-domain $\mathcal H$, and that it is self-adjoint and trace class.
 
\begin{customCondition}{$G_1$}\label{Cond:2Tail}
For some orthonormal basis $(\phi_i)_{i\geq 1}$ of $\mathcal H$ we have $\sum_{i\geq 1} \sigma(\phi_i,\phi_i)<\infty$.
\end{customCondition} 

A standard exercise shows that if the condition above holds for one orthonormal basis, then it holds for every orthonormal basis. 
 
The fact that $T_{\sigma}$ is trace class and self-adjoint implies that there exists an orthonormal basis of $\mathcal H$, say $(\phi_i)_{i\geq 1}$, such that for every $i\geq 1$ we get $T_{\sigma}\phi_i = \lambda_i \phi_i$, and the sum of the eigenvalues $\lambda_i$ is bounded (they are non-negative as $T_{\sigma}$ is self-adjoint with finite multiplicity, except the potential eigenvalue 0). This will be important to define the (potential) limit of $S_n$.
 
Finally, to ensure that $S_n$ actually converges in distribution to a limiting functional we require the following tightness condition:

\begin{customCondition}{$G_2$}\label{Cond:1bound}
For some orthonormal basis $(\phi_i)_{i\geq 1}$ of $\mathcal H$, and for any $\varepsilon>0$,  we have that
\begin{align*}
    \lim_{i\to \infty} \limsup_{n\to \infty} \Prob(\|S_n\circ P_{V_i^{\perp}}\|_{\mathcal H \to \R}^2\geq \varepsilon) = 0,
\end{align*}
where $V_i$ is the span of $\phi_1,\ldots, \phi_i$, and $ P_{V_i^{\perp}}$ is an orthogonal projection onto $V_i^\perp$.  
\end{customCondition}
We remark that $\|S_n\circ P_{V_i^{\perp}}\|_{\mathcal H \to \R}^2 = \sum_{k=i+1}^{\infty} S_n(\phi_k)^2$, which is a useful expression to bound the probability above in conjunction with the Markov inequality. It can also we shown that if \Cref{Cond:1bound} holds for one basis, then it holds for any basis of $\mathcal H$, however, this takes some work and for completeness we provide a proof in \Cref{sec:aux}.

\begin{theorem}\label{thm:ConvergenceSumChiSq}
Let $(S_n)_{n\geq 1}$ be a sequence of bounded linear test-statistics satisfying \Cref{Cond:0bilinear,Cond:1bound,Cond:2Tail}.  Define the random functional $$S(\cdot) = \sum_{i=1}^{\infty} \sqrt{\lambda_i}\InerH{\phi_i}{\cdot}Z_i,$$ 
where  $(\lambda_i,\phi_i)_{i\geq1}$ are the eigenvalues and eigenvectors of the operator $T_{\sigma}:\mathcal{H}\to\mathcal H$ defined in \cref{eqn:Tsigma}, and  $(Z_i)_{i\geq1}$ are a collection of i.i.d. standard normal random variables. 

Then $S$ exists almost surely, i.e., the sum converges almost surely in $\mathcal H^\star$, and
\begin{align}
 S_n\overset{\mathcal D}{\to} S,  \qquad\text{and}\qquad \Psi_n= \|S_n\|^2_{\mathcal H \to \R}\overset{\mathcal D}{\to} \sum_{i=1}^{\infty}\lambda_i Z_i^2.
\end{align}
\end{theorem}
We also show that the conditions imposed in \Cref{thm:ConvergenceSumChiSq} are also necessary.

\begin{theorem}\label{thm:necessity}
Let $(S_n)_{n\geq1}$ be a sequence of bounded linear test-statistics in  $\mathcal H^\star$, and define $S(\cdot) = \sum_{i=1}^{\infty} \sqrt{\lambda_i} Z_i \InerH{\phi_i}{\cdot}$, where $(\lambda_i)_{i\geq0}$ are positive constants and $(\phi_i)_{i\geq 1}$ is an orthonormal basis of $\mathcal H$. Suppose that $S$ converges almost surely in $\mathcal H^\star$ (i.e. almost surely the truncated sums are  Cauchy sequences in $\mathcal H^\star$). Then, if $S_n\overset{\mathcal D}{\to} S,$
we have that \Cref{Cond:0bilinear,Cond:1bound,Cond:2Tail} hold.
\end{theorem}

We continue our analysis with the second case of interest, which occurs when for every $\omega \in \mathcal H$,  $S_n(\omega)$ converges a.s. or in probability to a constant value $c(\omega)$. This is a typical situation under the alternative, and is much simpler to describe.

\begin{theorem}\label{Thm:alternati}
Consider a sequence $(S_n)_{n\geq 1}$ of linear test-statistics, and suppose that for each $\omega\in \mathcal H$ we have  $S_n(\omega) \overset{a.s.}{\to} c(\omega)$, where $c:\mathcal H \to \R$ is a deterministic functional. Define $c^\star$ as  $c^\star=\sup_{\unitball} c(\omega)$, and suppose $c^{\star} <\infty$. Then
\begin{align*}
\lim_{n\to \infty}\Psi_n =  \lim_{n\to \infty}\sup_{\unitball}S_n(\omega)^2= (c^\star)^2 \text{ }a.s.
\end{align*}
if and only if for some basis $(\phi_i)_{i\geq 1}$ of $\mathcal H$ we have $\lim_{u\to \infty} \limsup_{n\to \infty } \sum_{i=u+1}^{\infty}S_n(\phi_i)^2 = 0$ a.s.

Moreover, if for every $\omega\in \mathcal H$ we have $S_n(\omega)\overset{\Prob}{\to} c(\omega)$, then $\Psi_n\overset{\Prob}{\to} (c^\star)^2$ if and only if for some basis $(\phi_i)_{i\geq1}$ of $\mathcal H$ it holds $\lim_{u\to \infty} \limsup_{n\to \infty } \Prob(\sum_{i=u+1}^{\infty}S_n(\phi_i)^2\geq\varepsilon)=0$ for all  $\varepsilon>0$.
\end{theorem}

We remark that if the condition of \Cref{Thm:alternati} holds for one basis, then it holds for all basis at the same time. Note such condition is essentially \Cref{Cond:1bound}.

\subsection{Proofs of \Cref{thm:ConvergenceSumChiSq,thm:necessity,Thm:alternati}}

The proof of the theorems require the following basic properties of the operator $T_{\sigma}$.
\begin{proposition}\label{prop:PropertiesofT} Let $\sigma:\mathcal{H}\times\mathcal{H}\to\R$ be a continuous bilinear form satisfying  \Cref{Cond:2Tail}. Then, the operator $T_{\sigma}$ defined in \cref{eqn:Tsigma} satisfies:
\begin{enumerate}
    \item\label{property:HsubL}  $T_{\sigma}f \in \mathcal H$ for any $f \in \mathcal H$. 
    \item\label{property:isometryHs} $\InerH{T_{\sigma}f}{\omega}= \sigma(f,\omega)$  for any $f,\omega \in \mathcal H$.
    \item\label{property:TisSelfAdjoint} $T_{\sigma}:\mathcal H\to \mathcal H$ is self-adjoint.
    \item \label{property:compact} $T_\sigma$ is trace-class.
\end{enumerate}
\end{proposition}

\begin{proof} We start by proving \cref{property:HsubL}. For any $f\in\mathcal{H}$, define $A_f:\mathcal H \to \R$ by $A_f\omega=\sigma(f,\omega)$ for all $\omega\in\mathcal{H}$. Clearly $A_f$ is linear since $\sigma$ is bilinear. Moreover, by continuity of $\sigma$, we have that $\sup_{\omega\in\mathcal{H}:\LH{\omega}=1}\sigma(\omega,\omega)<\infty$, and thus, by the Cauchy-Schwarz inequality,
\begin{align*}
    \sup_{\omega\in\mathcal{H}:\LH{\omega}=1}|A_f\omega|^2=\sup_{\omega\in\mathcal{H}:\LH{\omega}=1}|\sigma(f,\omega)|^2&\leq \sup_{\omega\in\mathcal{H}:\LH{\omega}=1}\sigma(\omega,\omega)\sigma(f,f)<\infty.
\end{align*}
Then, by the Riesz's representation theorem, for all $f\in\mathcal{H}$, there exists unique element $\xi_f\in \mathcal H$ such that for all $\omega \in \mathcal H$, $A_f\omega=\sigma(f,\omega) = \InerH{\xi_f}{\omega}$ holds. By choosing $\omega = K_x$ for any $x\in\mathcal{X}$, we get 
\begin{align}
    (T_{\sigma}f)(x) = \sigma(f,K_x) =A_{f}K_x= \InerH{\xi_f}{K_x} = \xi_f(x),
\end{align}
and thus $T_{\sigma}f = \xi_f\in \mathcal H$.

To prove \cref{property:isometryHs}, observe that \cref{property:HsubL} yields,  $T_{\sigma}f = \xi_f$ and thus $\InerH{T_{\sigma}f}{\omega} = \InerH{\xi_f}{\omega}  = \sigma(f,\omega)$. 

For \cref{property:TisSelfAdjoint}, by  the previous item and the symmetry of $\sigma$,  we have that 
\begin{align*}
\InerH{T_\sigma f}{\omega}=\sigma(f,\omega) = \sigma(\omega,f) =\InerH{T_{\sigma}\omega}{ f} = \InerH{ f}{T_{\sigma}\omega}.
\end{align*}

Finally, to prove \cref{property:compact}, consider an orthonormal basis $(\phi_i)_{i\geq 1}$ of $\mathcal{H}$, and observe  that
\begin{align*}
    \text{Trace}(T_\sigma)=\sum_{i=1}^\infty\InerH{\phi_i}{T_\sigma \phi_i}=\sum_{i=1}^\infty\sigma(\phi_i,\phi_i)<\infty,
\end{align*}
where the second equality is due to \cref{property:isometryHs}, and the last inequality is due to \Cref{Cond:2Tail}. We deduce then that $T_\sigma$ is trace-class.
\end{proof}

\subsubsection{{{Proof of \Cref{thm:ConvergenceSumChiSq,thm:necessity}}}}

\begin{proof}[Proof of \Cref{thm:ConvergenceSumChiSq}] By \Cref{prop:PropertiesofT} we have that $T_{\sigma}:\mathcal{H}\times\mathcal{H}\to\R$ defined in \cref{eqn:Tsigma} is self-adjoint and trace class in $\mathcal H$. Then, by the spectral theorem there exists an orthonormal basis $(\phi_i)_{i\geq 1}$ of $\mathcal H$ such that $T_{\sigma}\phi_i = \lambda_i \phi_i$ for all $i\geq 0$. Recall that the basis is countable since $\mathcal H$ is separable, and that the eigenvalues are all non-negative since $\lambda_i=\InerH{T_{\sigma}\phi_i}{\phi_i}=\sigma(\phi_i,\phi_i)\geq 0,$
by \cref{property:isometryHs} of \Cref{prop:PropertiesofT} (recall that $\sigma(\phi_i,\phi_i)$ is a variance, so it is non-negative). Note as well that $\sum_{i=1}^{\infty}\lambda_i<\infty$ as $T_{\sigma}$ is trace class.

Note that $S_n(\omega) = \sum_{i=1}^{\infty} \InerH{\phi_i}{\omega} S_n(\phi_i)$, and recall that the functional $S$ is defined by $S(\omega) = \sum_{i=1}^{\infty} \InerH{\phi_i}{\omega} \sqrt{\lambda_i}Z_i$, where $Z_i$ are i.i.d. standard normal random variables. Observe that $S$ is well-defined since $\sum_{i=1}^{\infty} \lambda_i Z_i^2$ converges a.s. due to \Cref{lemma:lambda3ST} (in \Cref{sec:aux}) since we have that $\sum_{i=1}^{\infty}\lambda_i<\infty$. Our proof also requires the definition of a partial version of $S_n$ and $S$, given by
\begin{align*}
    S_n^U = \sum_{i=1}^U\InerH{\phi_i}{\cdot} S_n(\phi_i)\quad \text{ and } \quad S^U(\omega) = \sum_{i=1}^{U} \InerH{\phi_i}{\omega} \sqrt{\lambda_i}Z_i.
\end{align*}

We will show that $S_n \overset{\mathcal D}{\to} S$ via an application of Theorem 3.2 of \citet{billingsley2013convergence}, which requires the following properties:
\begin{enumerate}[label = \roman*.]
    \item $ S_n^U\overset{\mathcal D}{\to} S^U$ as $n \to \infty$
    \item $S^U \overset{\mathcal D}{\to} S$ as $U \to \infty$
    \item for all $\varepsilon>0$, 
    \begin{align}\label{eqn:proofTheo1cond3}
       \lim_{U\to \infty}\limsup_{n\to \infty}\Prob\left(\|S_n-S_{n}^U\|_{\mathcal H \to \R}>\varepsilon\right) = 0.
    \end{align}
\end{enumerate}
For the first property, \Cref{Cond:0bilinear} tells us that the random vector  $\boldsymbol{S}_n^U = (S_n(\phi_1),\ldots,S_n(\phi_U))$ in $\R^U$ is such that $\boldsymbol{S}_n^U\overset{\dist}{\to}N(0,\boldsymbol{\Sigma})$, where   $\boldsymbol{\Sigma}_{ij}=\sigma(\phi_i,\phi_j)$ for any $i,j\in\{1,\ldots,U\}$. Moreover, by \Cref{prop:PropertiesofT}, it holds that
\begin{align*}
    \boldsymbol{\Sigma}_{ij}=\sigma(\phi_i,\phi_j)&=\InerH{T_\sigma\phi_i}{\phi_j}= \lambda_i\InerH{\phi_i}{\phi_j}= \lambda_i \delta_{ij}\end{align*}
where $\delta_{ij}=1$ if $i=j$, otherwise it is $0$. Then $\boldsymbol{\Sigma}=\text{diag}(\lambda_1,\ldots,\lambda_U)$, and thus  $\boldsymbol{S}_n^U$ converges in distribution to the vector $(\sqrt{\lambda_1}Z_1,\ldots,\sqrt{\lambda_U}Z_U)$, where $Z_1,\ldots,Z_U$ are independent and identically distributed standard normal random variables. Consequently, the continuous mapping theorem and the continuity of the transformation $\R^U \to \mathcal H^\star$ given by $\x \to \sum_{i=1}^U x_i\InerH{\phi_i}{\cdot}$ yields property i.

The second property follows immediately from the definition of $S_n^U$. Finally, for the last property, denote $V_U$ be the subspace generated by $\langle \phi_1,\ldots, \phi_U\rangle$, then note that $S_n -S_n^U=S_n\circ P_{V_U^{\perp}}$, so iii. follows directly from \Cref{Cond:1bound}.

Since the three conditions have been proved, Theorem 3.2 of \citet{billingsley2013convergence} yields that $S_n\overset{\mathcal D}{\to} S$, and by the continuous mapping theorem, we get  $\|S_n\|_{\mathcal H \to \R}^2 \overset{\mathcal D}{\to} \|S\|_{\mathcal H \to \R}^2 = \sum_{i=1}^{\infty} \lambda_i Z_i^2$.
\end{proof}

\begin{proof}[Proof of \Cref{thm:necessity}]
Since $S$ converges almost surely in $\mathcal H^\star$ we have that $\|S\|_{\mathcal{H}\to\R}^2 = \sum_{i\geq1} \lambda_i Z_i^2$ converges almost surely in $\R$, and so by \Cref{lemma:lambda3ST} (in \Cref{sec:aux}) we have $\sum_{i\geq1} \lambda_i<\infty$.

Let's verify \Cref{Cond:0bilinear}. For any given $\omega_1,\ldots, \omega_k \in \mathcal H$ the transformation $\mathcal H^\star \to \R^k$ given by $Q \to (Q(\omega_1),\ldots, Q(\omega_k))$ is continuous in $\mathcal H^\star$, then by the continuous mapping theorem
\begin{align}
    (S_n(\omega_1),\ldots,S_n(\omega_k)) \overset{\mathcal D}{\to} \left( \sum_{i=1}^\infty \sqrt{\lambda_i} \InerH{\phi_i}{\omega_1} Z_i,\ldots, \sum_{i=1}^\infty \sqrt{\lambda_i} \InerH{\phi_i}{\omega_k} Z_i \right).
\end{align}

The sums on the right-hand side term converge almost surely in $\R$ by the Doob's martingale convergence theorem: indeed, for any $\omega \in \mathcal H$, $M_t = \sum_{i=1}^t \sqrt{\lambda_i} Z_i \InerH{\phi_i}{\omega}$ is a 0 mean martingale with second moment $\sum_{i=1}^t \lambda_i \InerH{\phi_i}{\omega}^2 \leq \|\omega\|^2\sum_{i\geq1} \lambda_i$, i.e. uniformly bounded second moment, so the sum converges almost surely. We can now easily see that the right-hand side term has multi-variate Gaussian distribution of mean 0 and covariance matrix $\Sigma$ of $k\times k$ with entries $\Sigma_{p,q} = \sigma(\omega_p, \omega_q) = \sum_{i=1}^\infty\lambda_i \InerH{\phi_i}{\omega_p}\InerH{\phi_i}{\omega_q}$.

To verify \Cref{Cond:2Tail} just observe that $\sigma(\phi_i, \phi_i) = \lambda_i$, thus the condition holds since $\sum_{i\geq 1} \lambda_i<\infty$. 

Finally, to prove \Cref{Cond:1bound} note that for any subspace $V$ of $\mathcal H$, the transformation $\mathcal H^\star \to \R$ given by $Q\to \|Q\circ P_V\|_{\mathcal H\to \R}$ is continuous, where $P_V$ is the orthogonal projection onto $V$. Now, let $V_i$ be the span of $\phi_1,\ldots, \phi_i$, then
$$\|S_n\circ P_{V_i^{\perp}}\|_{\mathcal H\to \R} \overset{\mathcal D}\to \|S\circ P_{V_i^{\perp}}\|_{\mathcal H\to \R},$$
by the continuous mapping theorem and the fact that $S_n\overset{\mathcal D}{\to} S$. This implies that for any $\varepsilon> 0$ we have
\begin{align*}
    \limsup_{n\to \infty} \Prob\left(\|S_n\circ P_{V_i^{\perp }}\|_{\mathcal H\to \R}\geq\varepsilon\right) =  \Prob\left(\|S \circ P_{V_i^{\perp}}\|_{\mathcal H\to \R}\geq\varepsilon\right),
\end{align*}
and by taking limit when $i$ tends to infinity we get \Cref{Cond:1bound}, since $\|S\circ P_{V_i^{\perp}}\|^2 = \sum_{j=i+1}^{\infty} \lambda_j Z_j^2\to 0$ a.s. as $i$ tends to infinity.

\end{proof}

\subsubsection{Proof of \Cref{Thm:alternati}}

The proof of \Cref{Thm:alternati} is much simpler than the previous ones, and can be worked out directly from first principles.

\begin{proof}[Proof of \Cref{Thm:alternati}] Let's assume that $\lim_{u\to \infty} \limsup_{n\to \infty } \sum_{i=u+1}^{\infty}S_n(\psi_i)^2=0$ holds for some basis $(\psi_i)_{i\geq1}$, we will prove that $\Psi_n\to (c^\star)^2$ a.s.

Start by noting that the limit functional $c:\mathcal H \to \R$ is linear since $S_n$ is linear, and by the hypothesis it is also bounded. Then, by the Riesz representation theorem, there exists $\xi \in \mathcal H$ such that $c(\omega) = \InerH{\xi}{\omega}$. Let $\phi_1 = \xi/\LH{\xi}$, and complete an orthonormal basis $\phi_1,\phi_2,\ldots$ of $\mathcal H$. Then note that 
\begin{align*}
c(\phi_1)=\LH{\xi}\qquad\text{and}\qquad c(\phi_i)=\LH{\xi}\InerH{\phi_1}{\phi_i}=0,\quad \text{for all }i\geq 2.    
\end{align*}

Now, since $S_n:\mathcal{H}\to\R$ is linear and bounded, there exists $\xi_n \in \mathcal H$ such that
\begin{align*}
    \sup_{\omega\in \mathcal{H}:\LH{\omega}=1}S_n(\omega)^2 &= \LH{\xi_n}^2 =  \sum_{i=1}^{\infty} \InerH{\xi_n}{ \phi_i}^2 = \sum_{i=1}^{\infty} S_n(\phi_i)^2.
\end{align*}
Note that $S_n(\phi_1)^2 \to c(\phi_1)^2$, and that $S_n(\phi_i)^2 \to c(\phi_i)^2=0$ a.s. for all $i\geq 2$.  Then,  we just need to show that $\lim_{n\to \infty}\sum_{i=2}^{\infty} S_n(\phi_i)^2 = 0$ a.s. The latter follows from a simple computation. Indeed, let $u$ be a positive integer and write $
    \sum_{i=2}^{\infty} S_n(\phi_i)^2 = \sum_{i=2}^{u} S_n(\phi_i)^2+ \sum_{i=u+1}^{\infty} S_n(\phi_i)^2$, so
\begin{align*}
    \limsup_{n\to \infty} \sum_{i=2}^{\infty} S_n(\phi_i)^2 &= \limsup_{n\to \infty}\sum_{i=2}^{u} S_n(\phi_i)^2+\limsup_{n\to \infty} \sum_{i=u+1}^{\infty} S_n(\phi_i)^2=0+\limsup_{n\to \infty} \sum_{i=u+1}^{\infty} S_n(\phi_i)^2.
\end{align*}
Then by taking limit when $u$ tends to infinity we get $\limsup_{n\to \infty} \sum_{i=2}^{\infty} S_n(\phi_i)^2 = 0$.

Now we prove the converse. Suppose that $\Psi_n\to (c^\star)^2$ a.s. Consider the orthonormal basis $(\phi_i)_{i\geq1}$ as above. Then, note that since almost surely $\lim_{n\to \infty} \sum_{i=1}^{\infty} S_n(\phi_i)^2 = (c^\star)^2$ and $\lim_{n\to \infty} S_n(\phi_1)^2 = (c^\star)^2$, it holds
$$\lim_{n\to \infty}\sum_{i=2}^{\infty} S_n(\phi_i)^2 \to 0 \text{ } a.s.$$
thus $\lim_{u \to \infty} \lim_{n\to \infty}\sum_{i=u+1}^{\infty} S_n(\phi_i)^2 = 0$ a.s.

The analogous result for convergence in probability follows by using the same arguments.
\end{proof}

\section{Application to Hypothesis Testing}\label{sec:examples}

In this section we show how the results of the previous section can be applied in the analysis of kernel-based tests. We begin by showing that it is possible to easily evaluate $\Psi_n$ as long as evaluating $S_n(\omega)$ is easy. For that, our next proposition gives a closed-form expression for the kernelised statistic $\Psi_n$ in terms of $S_n$. 

\begin{proposition}\label{prop:Vstatshape} Suppose that $S_n$ is a bounded linear statistic, then
\begin{align*}
    \Psi_n= \|S_n\|_{\mathcal H \to \R}^2 = S_n^1S_n^2K,
\end{align*}
where $S_n^iK$ with $i\in\{1,2\}$ denotes the application of the transformation $S_n:\mathcal{H}\to\R$ to the i-th coordinate of the function $K$. 
\end{proposition}
\begin{proof} 
By the Riesz representation theorem, there exists a unique element $\xi_n\in\mathcal{H}$ such that $S_n(\omega)=\InerH{\xi_n}{\omega}$ for all $\omega\in\mathcal{H}$. Then
 \begin{align*}
    \Psi_n=\sup_{\unitball} S_n(\omega)^2=\sup_{\unitball} \InerH{\xi_n}{\omega}^2=\LH{\xi_n}^2=\InerH{\xi_n}{\xi_n}.
\end{align*}
Note that by the definition of $\xi_n$ we have $\InerH{\xi_n}{\xi_n}=S_n(\xi_n)$, and $\xi_n(x)=\InerH{\xi_n}{K_x}=S_n(K_x)$. Thus we conclude that
$$\Psi_n=S_n(\xi_n)=S_n(S_n(K_\cdot))=S_n^1S_n^2K.$$
\end{proof}

It is important to notice that for our analysis we do not actually need to write $\Psi_n$ in the closed-form expression of \cref{prop:Vstatshape}, but in practice this is important for the numerical evaluation of the test-statistic and henceforth the computer implementation of the test.

\subsection{MMD for the two-sample problem}\label{sec:MMD}
We first study the well-known Maximum-Mean-Discrepancy ($\MMD$) on RKHS's introduced in \cite{gretton2006kernel}.
The $\MMD$ is a test statistic which is  used to determine if two distributions, $F_0$ and $F_1$, are the same by means of $n= n_0+n_1$ independent random samples $(X_i)_{i=1}^{n_0}\sim F_0$ and  $(Y_i)_{i=1}^{n_1}\sim F_1$. The $\MMD$ is given by
\begin{align}
    \MMD\left(\widehat F_{0},\widehat F_{1}\right)&=\sup_{\omega\in\mathcal{H}:\LH{\omega}=1}\frac{1}{n_0}\sum_{i=1}^{n_0}\omega(X_i)-\frac{1}{n_1}\sum_{i=1}^{n_1}\omega(Y_i),\label{eqn:MMD}
\end{align}
where $\mathcal{H}$ is a RKHS with reproducing kernel given by $K$. In the above equation $\widehat F_0$ and $\widehat F_1$ represent the empirical distributions of $(X_i)_{i=1}^{n_0}$ and $(Y_i)_{i=1}^{n_1}$, respectively, so we avoid writing $\MMD((X_i)_{i=1}^{n_0},(Y_i)_{i=1}^{n_1})$. To design a testing procedure it is fundamental to study the distribution of the $\MMD$. In \cite{gretton2006kernel} and subsequent works \cite{gretton2007kernel,gretton2009Fast,gretton2012kernel,Chwialkowski2014Wild} the authors study the distribution of \cref{eqn:MMD} by using the following representation
\begin{align*}
 \MMD\left(\widehat F_{0},\widehat F_{1}\right)^2 = \frac{1}{n_0^2} \sum_{i=1}^{n_0}\sum_{j=1}^{n_0} K(X_i,X_j) - \frac{2}{n_0n_1}\sum_{i=1}^{n_0}\sum_{j=1}^{n_1} K(X_i,Y_j)+ \frac{1}{n_1^2} \sum_{i=1}^{n_1}\sum_{j=1}^{n_1} K(Y_i,Y_j),   
\end{align*}
which they obtain from the fact that the supremum is taken over the unit ball of an RKHS. From the previous equation, it can be deduced  that the square of the $\MMD$ is a two-sample $V$-statistic, and thus it can be studied under null and alternative hypothesis by using tools from the theory of $U$ and $V$ statistics. A drawback however is that the analysis of two-sample $V$-statistics can be lengthy, as it requires the computation of the first and second moment of it, i.e. terms involving up to 4 different $X_i'$s and $Y_j'$s.

The main asymptotic result in this setting is the following. Suppose that there are no vanishing groups, that is, $n_0/n\to \rho_0\in (0,1)$ as $n = n_0+n_1\to \infty$, and that the kernel satisfies some integrability conditions. Then, under the null hypothesis, $n\MMD(\widehat F_{0},\widehat F_{1})^2$ converges in distribution to a potentially infinite sum of independent weighted $\chi^2_1$ random variables when $n$ tends to infinity. Finding rejections regions for the $\MMD$ is not  simple, and various approaches had been proposed for such a task. In \cite{gretton2012kernel}, the authors proposed to fit Pearson curves using the first four moments of $\MMD(\widehat F_{0},\widehat F_{1})^2$ which is a computationally expensive procedure. In \cite{gretton2009Fast}, it is proposed to estimate the first $n$ weights of the infinite combination of independent $\chi^2_1$ random variables. A Wild-Bootstrap approach was proposed by \cite{Chwialkowski2014Wild}, and it is currently the standard approach used in kernel-based tests.

In the following subsections we will show that the results of \Cref{sec:convergence} can be used to obtain the results presented in  \cite{gretton2006kernel}, \cite{gretton2012kernel} and  \cite{Chwialkowski2014Wild} but using a much simpler and shorter analysis. 

\subsubsection{Analysis under the null hypothesis}
For our analysis we assume the same setting presented at the beginning of \Cref{sec:MMD}, and additionally, we assume the following condition holds true.
\begin{condition}\label{Condi: MMD example}
Assume that $K$ is bounded i.e. there is a constant $C>0$ such that for all $x,y\in \mathcal X$,  $|K(x,y)|\leq C$. Also, assume that there are no vanishing groups, i.e. $\lim_{n\to\infty} n_0/n=\rho_0$ and $\lim_{n\to\infty} n_1/n=\rho_1$, with $\rho_0,\rho_1\in (0,1)$.
\end{condition}
We remark that the condition over the kernel is for simplicity as it will ease some calculations. Indeed, note that by this condition the functions in the unit ball of $\mathcal H$ are uniformly bounded by $\sqrt{C}$ since $|\omega(x)|^2 \leq \LH{K_x}^2\LH{\omega}^2 = K(x,x)\leq C$.

Based on the data $(X_i)_{i=1}^{n_0}$ and $(Y_i)_{i=1}^{n_1}$, define
\begin{align*}
     \Psi_n= \sup_{\unitball}S_{n}(\omega)^2  \quad \text{ where }\quad S_{n}(\omega)=\sqrt{n}\left(\frac{1}{n_0}\sum_{i=1}^{n_0}\omega(X_i)-\frac{1}{n_1}\sum_{i=1}^{n_1}\omega(Y_i)\right).
\end{align*}
Clearly, $S_{n}$ is a linear statistic on $\omega$, and $\Psi_n = n\MMD(\widehat F_0, \widehat F_1)^2$. Moreover, by using that $|\omega(x)|^2\leq C$, we get that $|S_{n}(\omega)| \leq 2\sqrt{Cn}$, and thus we deduce that $S_{n}$ is a bounded linear test-statistic for each fixed $n_0$ and $n_1$. We will use \Cref{thm:ConvergenceSumChiSq} to obtain the asymptotic limit distribution of $\Psi_n$. The limit distribution of $\Psi_n$ will be characterised by the bilinear form $\sigma:\mathcal H \times \mathcal H \to \R$ given by
\begin{align}
\sigma(\omega,\widetilde{\omega})=\frac{1}{\rho_0\rho_1}\int\left(\omega(x)-\int\omega(y)dF_0(y)\right)\left(\widetilde{\omega}(x)-\int\widetilde{\omega}(y)dF_0(y)\right)dF_0(x).\label{eqn:sigma}    
\end{align}

\begin{lemma}\label{lemma:twosamplenormal}
Suppose that \Cref{Condi: MMD example} holds. Then, under the null hypothesis we have that for any $\omega_1,\ldots, \omega_{\ell} \in \mathcal H$,
\begin{align*}
\boldsymbol{S}_{n}=(S_{n}(\omega_1),S_{n}(\omega_2),\ldots,S_{n}(\omega_\ell))\overset{\mathcal D}{\to} N(0,\Sigma),
\end{align*}
as $n$ grows to infinity, where $\Sigma_{i,j} = \sigma(\omega_i,\omega_j)$ for any $i,j\in[\ell]$.

Moreover, for any $\omega\in \mathcal H$,
\begin{align*}
\E(S_{n}(\omega)) = 0\quad \text{ and}\quad \E(S_{n}(\omega)^2) = \frac{n^2}{n_0n_1}\int \left(\omega(x)-\int \omega(y)dF_0(y)\right)^2 dF_0(x).
\end{align*}
\end{lemma}
\begin{proof}
By \Cref{Condi: MMD example}, $K$ is bounded. Then, by the Central Limit Theorem and the fact that a linear combination of independent normal random variables is normal, we get that  $\boldsymbol{S}_{n}\overset{\dist}{\to}N_\ell(0,\boldsymbol{\Sigma}),$
as $n\to\infty$, where $\boldsymbol{\Sigma}_{i,j}=\sigma(\omega_i,\omega_j)$ for any $i,j\in[\ell]$. The second part of the lemma is a simple computation, using the fact that the random variables are independent, and that, by \Cref{Condi: MMD example}, $\lim_{n\to\infty} n_0/n=\rho_0$ and $\lim_{n\to\infty} n_1/n=\rho_1$ where $\rho_0,\rho_1\in(0,1)$.
\end{proof}

\begin{lemma}\label{lemma:twosampleTraceClass} 
Let $(\phi_i)_{i\geq 1}$ be an orthonormal basis of $\mathcal H$, then for any probability measure $F$ we have
\begin{align*}
    \sum_{i=1}^\infty \int \phi_i(x)^2dF(x) = \int K(x,x)dF(x)<\infty
\end{align*}
\end{lemma}
\begin{proof}
Since $F$ is a probability measure and the kernel is bounded, the result follows immediately since $$\sum_{i\geq 1}\int\phi_i(x)^2dF(x) = \sum_{i\geq 1} \InerH{\phi_i}{K_x}^2 = \LH{K_x}^2=K(x,x).$$
\end{proof}

\begin{proposition}
Suppose that  \Cref{Condi: MMD example} holds. Then, under the null hypothesis  $\Psi_n\overset{\dist}{\to}\Psi:=\sum_{i=1}^\infty\lambda_iZ_i^2$, where $Z_i$ are i.i.d. standard normal random variables, and $\lambda_i$ are the eigenvalues of the operator $T_\sigma$ associated with the covariance $\sigma$.
\end{proposition}
\begin{proof}
We proceed to verify \Cref{Cond:0bilinear,Cond:1bound,Cond:2Tail} to apply \Cref{thm:ConvergenceSumChiSq}.  Since $F_0  = F_1$ we just write $F$.
\Cref{Cond:0bilinear} follows immediately from \Cref{lemma:twosamplenormal}. For \Cref{Cond:2Tail}, let $(\phi_k)_{k\geq 1}$ be an orthonormal basis of $\mathcal{H}$, then by the definition of $\sigma$ and \Cref{lemma:twosampleTraceClass}, we have
\begin{align*}
    \sum_{i=1}^\infty\sigma(\phi_i,\phi_i)\leq  \sum_{i=1}^\infty\int \phi_i(x)^2 dF(x)=\int K(x,x)dF(x)<\infty.
\end{align*}
Finally, let's verify \Cref{Cond:1bound}. Let $V_i$ be the span by $\phi_1,\ldots,\phi_i$. Note that $\|S_{n} \circ P_{V_i^\perp}\|_{\mathcal{H}\to\R}^2=\sum_{k=i+1}^\infty S_{n}(\phi_k)^2$, thus the Markov inequality and \Cref{lemma:twosamplenormal}, yields that for any $\varepsilon> 0$:
\begin{align*}
    \Prob\left( \|S_{n}\circ P_{V_i^\perp}\|_{\mathcal{H}\to\R}^2\geq \varepsilon\right)\leq \sum_{k=i+1}^\infty \varepsilon^{-1}\E\left(S_{n}(\phi_k)^2\right)
    &\leq  \frac{n^2}{\varepsilon n_0n_1}\sum_{k=i+1}^\infty \int \phi_k(x)^2 dF(x).
\end{align*}
By \Cref{lemma:twosampleTraceClass} we have $\lim_{i\to \infty}\sum_{k=i+1}^\infty \int \phi_k(x)^2 dF(x)= 0$, concluding that 
\begin{align*}
   \lim_{i\to \infty} \limsup_{n\to\infty} \Prob\left( \|S_{n}\circ  P_{V_i^\perp}\|_{\mathcal{H}\to\R}^2\geq \varepsilon\right)= 0. 
\end{align*}
\end{proof}

\subsubsection{Analysis under the alternative hypothesis}\label{sec:alternativetwosample}

\begin{proposition}
Assume that \cref{Condi: MMD example} holds. Then, under the alternative hypothesis, it exists $c^\star\in \R$ such that $\frac{1}n\Psi_n\to (c^\star)^2$ a.s.
\end{proposition}
\begin{proof}
We shall verify the conditions of \Cref{Thm:alternati} to prove almost sure convergence. By the law of large numbers we get
\begin{align*}
\frac{1}{\sqrt{n}}S_{n}(\omega) =  \frac{1}{n_0}\sum_{i=1}^{n_0}\omega(X_i)-\frac{1}{n_1}\sum_{i=1}^{n_1}\omega(Y_i) \overset{a.s}{\to} c(\omega):= \int \omega(x)dF_0(x)-\int \omega(x)dF_1(x).    
\end{align*}
Let $c^\star = \sup_{\unitball}c(\omega)$, and note that $c^\star <\infty$ since the functions $\omega$ in the unit ball of $\mathcal H$ are bounded as the kernel is bounded.

Now, let $(\phi_k)_{k\geq 1}$ be an orthonormal basis of $\mathcal H$, then
$$n^{-1}\E\left(S_n(\phi_k)^2\right)\leq 2\E\left(\left(\frac{1}{n_0}\sum_{j=1}^{n_0} \phi_k(X_j)\right)^2\right)  +2\E\left(\left(\frac{1}{n_1}\sum_{j=1}^{n_1} \phi_k(Y_j)\right)^2\right) \leq 2\E(\phi_k(X_1)^2) + 2\E(\phi_k(Y_1)^2)$$
Therefore,
\begin{align*}
    \sum_{k=i+1}^{\infty}n^{-1}\E(S_n(\phi_k)^2) \leq   \sum_{k=i+1}^{\infty} \int \phi_k(x)^2(dF_0+dF_1)(x),
\end{align*}
and the later tends to 0 when $i$ grows to infinity due to \Cref{lemma:twosampleTraceClass}. With the conditions of \Cref{Thm:alternati} proven we conclude that $\frac{1}{n}\Psi_n \to (c^\star)^2$ a.s.
\end{proof}

We remark that in order to ensure that $c^\star\neq 0$, a sufficient condition is that the RKHS $\mathcal H$ has the property of being $c_0$-universal \cite{sriperumbudur2011universality}: let $\mathcal X$ be a locally compact Hausdorff space (e.g. $\R^d$), then a RKHS $\mathcal H$ with reproducing kernel  $K:\mathcal X \times \mathcal X \to \R$, such that $K$ is bounded and $K_x \in C_0(\mathcal X)$ (i.e. $K_x$ is a continuous function of $\mathcal X$ vanishing at infinity); we say that $\mathcal H$ (or rather the corresponding kernel $K$) is $c_0$-universal if and only if  every bounded signed-measure different from the zero measure satisfies $\sup_{\unitball}\int \omega(x)\mu(dx)> 0$. Now, the measure $\mu = F_0-F_1$ is clearly different from 0 under the alternative hypothesis, so $c^\star = \sup_{\unitball}\int \omega(x) dF_0(x)-dF_1(x) >0$ for a $c_0$-universal RKHS $\mathcal H$.

While the definition of $c_0$-universality is rather technical, most of the usual kernels/RKHS
such as the squared exponential kernel, the Laplacian kernel, and the rational quadratic kernel are $c_0$-universal \cite{sriperumbudur2011universality}.

\subsubsection{Wild bootstrap resampling scheme}\label{sec:twoSampleWB}
To build a proper test we need to find rejection regions. In particular, given a predetermined Type-I error $\alpha\in (0,1)$, it is enough to find the $(1-\alpha)$-quantile of $\Psi_n$ under the null and reject the null hypothesis whenever our test-statistic exceeds this quantile. Note however that the asymptotic null distribution of $\Psi_n$ is very complex and in most cases unknown. Hence, we use a Wild Bootstrap re-sampling scheme to approximate such distribution in order to approximate the desired quantile. For that define the Wild Bootstrap version of the linear test-statistic $S_{n}$ given by 
\begin{align*}
    S_{n}^W(\omega)=\sqrt{n}\left(\frac{1}{n_0}\sum_{i=1}^{n_0}W^X_i\omega(X_i)-\frac{1}{n_1}\sum_{i=1}^{n_1}W^Y_i\omega(Y_i)\right),
\end{align*} 
where $W_i^X=U_i-\frac{1}{n_0}\sum_{j=1}^{n_0} U_j$, $W_i^Y=V_i-\frac{1}{n_1}\sum_{j=1}^{n_1} V_j$, and $(U_i)_{i=1}^{n_0}$ and $(V_i)_{i=1}^{n_1}$ are two data sets of i.i.d. random variables with mean 0 and variance 1, that are sampled independently of the data. Then, we define $\Psi_n^W$ as the kernelised version of $S^W_{n}$, that is, $$\Psi_n^W=\sup_{\omega\in\mathcal{H}:\LH{\omega}^2=1}\left(S_{n}^W(\omega)\right)^2.$$
We will use our results, in particular \Cref{thm:ConvergenceSumChiSq}, to study the asymptotic behaviour of $\Psi_n^W$ under both, the null and alternative hypothesis. We recall that since we are bootstrapping, we assume that all the data points are fixed and that all the randomness is due to the random weights $W_i^X$ and $W_i^Y$, so to avoid a cumbersome notation we write $\Prob_D, \E_D$, etc., when conditioning on the data points (so we avoid expressing the conditional probabilities with large expressions). Additionally, we need a notion of convergence in distribution given the data. Given a sequence of random variables $Q_n$ taking values in a metric space, and another random variable $Q$ (taking values in the same metric space), then we say that $Q_n$ converges in distribution to $Q$ given the data points $D$ (denoted $Q_n\overset{\mathcal D_D}{\to} Q_n$) if and only if
\begin{align}\label{eqn:defDD}
 \Big|\E_D(f(Q_n))-\E_D(f(Q))\Big| \overset{\Prob}{\to} 0, \text{ for any bounded, uniformly continuous real-valued $f$.}    
\end{align}
As usual the limits are taken when $n$, the number of data points, tends to infinity. In all our examples the limit $Q$ is independent of $D$ and hence we have $\E_D(f(Q)) = \E(f(Q))$. Observe the only difference with the standard definition of convergence in distribution is that we are taking conditional expectation (which are random variables), so we take the limits in distribution.

The limiting behaviour of $\Psi_n^W$ is characterised by the bilinear form $\sigma^W:\mathcal H \times \mathcal H \to \R$ defined as
\begin{align}
    \sigma^W(\omega,\widetilde{\omega})&=\frac{1}{\rho_0}\int\left(\omega(x)-\int\omega(y)dF_0(y)\right)\left(\widetilde{\omega}(x)-\int\widetilde{\omega}(y)dF_0(y)\right)dF_0(x)\nonumber\\
    &\qquad+\frac{1}{\rho_1}\int\left(\omega(x)-\int\omega(y)dF_1(y)\right)\left(\widetilde{\omega}(x)-\int\widetilde{\omega}(y)dF_1(y)\right)dF_1(x).\label{eqn:bilinear}
\end{align}

\begin{lemma}\label{LemmaG0:TSW}
Suppose that \Cref{Condi: MMD example} holds. Then for any $\omega_1,\ldots, \omega_{\ell}\in \mathcal H$, it holds that 
\begin{align}
    \left(S^W_{n}(\omega_1), \ldots, S^W_{n}(\omega_{\ell})\right)\overset{\mathcal D_D}{\to} N(0,\boldsymbol{\Sigma}^W),\label{eqn:convergenceWBSn}
\end{align}
where for any $i,j\in[\ell]$, $\boldsymbol{\Sigma}_{ij}^W = \sigma^W(\omega_i,\omega_j)$.

Moreover, for any $\omega,\widetilde{\omega}\in\mathcal{H}$
\begin{align}
    \E_D(S_{n}^W(\omega))& = 0, \quad \text{and} \quad
    \Cov_D(S_{n}^W(\omega), S_{n}^W(\widetilde{\omega})) =\frac{n}{n_0^2}\sum_{i=1}^{n_0}\alpha_i^X\widetilde{\alpha}_i^X+\frac{n}{n_1^2}\sum_{i=1}^{n_1}\beta_i^X\widetilde{\beta}_i^X.\label{eqnWTS}
\end{align}
where $\alpha^X_i=\omega(X_i)-\frac{1}{n_0}\sum_{j=1}^{n_0} \omega(X_j)$ and $\beta^Y_i=\omega(Y_i)-\frac{1}{n_1}\sum_{j=1}^{n_1} \omega(Y_j)$, and $\widetilde{\alpha}^X_i$ and $\widetilde{\beta}^Y_i$ are similarly defined by replacing $\omega$ by $\widetilde{\omega}$.
\end{lemma}
\begin{proof}
By using the definition of the weights $(W_i^X)_{i=1}^{n_0}$ and $(W_i^Y)_{i=1}^{n_1}$, we can rewrite $S_{n}^W(\omega)$ as
\begin{align*}
    S_{n}^W(\omega)=\sqrt{n}\left(\frac{1}{n_0}\sum_{i=1}^{n_0}\alpha_i^X U_i-\frac{1}{n_1}\sum_{i=1}^{n_1}\beta^Y_iV_i\right).
\end{align*} 
Note that conditioned on the data, $S_{n}^W(\omega)$ is a linear combination of independent random variables. Thus, we can easily verify the results of \cref{eqnWTS}, and  that $\Var_D(S_n^W(\omega)) \to \sigma^W(\omega,\omega)$.  Now, since all the random variables involved in $S_{n}^W$ are bounded (recall that the reproducing kernel $K$ is bounded), the Lyapunov's central limit theorem holds,  proving that for any $\omega_1,\ldots, \omega_{\ell}\in \mathcal H$ we have $S_{n}^W(\omega)\overset{\mathcal D_D}{\to} N(0,\boldsymbol{\Sigma}^W)$.
\end{proof}

\begin{proposition}
Suppose that \Cref{Condi: MMD example} holds. Then $\Psi_n^{W}\overset{\mathcal D_D}{\to} \Psi^W:=\sum_{i=1}^{\infty} \lambda_i^W Z_i^2$, where $Z_i$ are i.i.d. standard normal random variables and $\lambda_i^W$ are the eigenvalues of the operator $T_{\sigma^W}$ associated with $\sigma^W$.
\end{proposition}
\begin{proof}
We shall prove \Cref{Cond:0bilinear,Cond:2Tail,Cond:1bound} to invoke  \Cref{thm:ConvergenceSumChiSq}. \Cref{Cond:0bilinear} follows immediately from \Cref{LemmaG0:TSW}. \Cref{Cond:2Tail} holds  since for any orthonormal basis $(\phi_i)_{i\geq 1}$ of $\mathcal H$ we have
\begin{align*}
    \sum_{i=1}^\infty\sigma^W(\phi_i,\phi_i)&\leq \sum_{k=1}^\infty \left(\frac{1}{\rho_0}\int \phi_k(x)^2dF_0(x)+\frac{1}{\rho_1}\int \phi_k(x)^2dF_1(x)\right)<\infty
\end{align*}
where the last inequality is due to \Cref{lemma:twosampleTraceClass}. Finally, for \Cref{Cond:1bound} let $V_i$ be the span of $\phi_1,\ldots, \phi_i$, then for any $\varepsilon\geq0$, 
\begin{align*}
    &\Prob_{D}\left(\left\|S_{n}^W\circ P_{V_i^\perp}^2\right\|^2_{\mathcal{H}\to\R}\geq \varepsilon\right) \leq \frac{1}{\varepsilon}\sum_{k=i+1}^\infty\E_D\left(S_{n}^W(\phi_k)^2\right)\\
    &= \frac{1}{\varepsilon}\sum_{k=i+1}^\infty\left(\frac{n}{n_0^2}\sum_{j=1}^{n_0}\left(\phi_k(X_j)-\frac{1}{n_0}\sum_{\ell=1}^{n_0} \phi_k(X_\ell)\right)^2+\frac{n}{n_1^2}\sum_{j=1}^{n_1}\left(\phi_k(Y_j)-\frac{1}{n_1}\sum_{\ell=1}^{n_1} \phi_k(Y_\ell)\right)^2\right)\\
   & \leq \frac{1}{\varepsilon}\sum_{k=i+1}^\infty\left(\frac{n}{n_0^2}\sum_{j=1}^{n_0}\phi_k(X_j)^2+\frac{n}{n_1^2}\sum_{j=1}^{n_1}\phi_k(Y_j)^2\right)= \frac{1}{\varepsilon}\left(\frac{n}{n_0^2}\sum_{j=1}^{n_0} \LH{P_{V_i^{\perp}}K_{X_j}}^2 + \frac{n}{n_1^2}\sum_{j=1}^{n_1}\LH{P_{V_i^{\perp}}K_{Y_j}}^2 \right).
\end{align*}

Then, by taking limit when $n\to\infty$, the law of large numbers yields
\begin{align*}
    \limsup_{n\to\infty}\Prob_{D}\left(\left\|S_{n}^WP_{V_i^\perp}\right\|^2_{\mathcal{H}\to\R}\geq \varepsilon\right)&\leq \frac{1}{\varepsilon}\left(\frac{1}{\rho_0} \int \LH{P_{V_i^{\perp}}K_{x}}^2F_0(x)+\frac{1}{\rho_1}\int \LH{P_{V_i^{\perp}}K_{x}}^2dF_1(x)\right)
\end{align*}
Finally, the above goes to $0$ when $i$ tend to infinity by \Cref{lemma:twosampleTraceClass}.

\end{proof}

With the above results we can build a test which is asymptotically correct: given one data set of $n$ data-points we can resample from $\Psi_n^W$ as much as we want (since we can resample the wild bootstrap weights), so we can find rejection regions with arbitrary large precision. Under the null hypothesis, $\Psi_n$ and $\Psi_n^W$ have the same limit distribution (as it can be easily verified that $\sigma = \sigma^W$ under the null), hence our test will have the right level $\alpha$. Moreover, under the alternative $\Psi_n\to \infty$ (assuming that $c^\star>0$) whereas  $\Psi_n^W$ converges in distribution, i.e. $\Psi_n^W = O_p(1)$, therefore under the alternative hypothesis the power of the test tends to 1 as the number of data points grows to infinity. 

\subsection{Kernel log-rank test for the two-sample problem}
We analyse another kernel test, now in the setting of survival analysis data. Consider a collection of $n$ i.i.d. right-censored data points $(X_i,\Delta_{i},g_i)_{i=1}^{n}$. In this setting,  $g_i\in\{0,1\}$ denotes a group-label, $X_i \in [0,\infty)$ is a time defined by $X_i=\min\{\widetilde{X}_i,C_i\}$ where $\widetilde{X}_i\sim F_{g_i}$ is a time of interest and $C_i\sim G_{g_i}$ denotes a nuisance censoring time, and  $\Delta_i=\ind_{\{\widetilde{X}_i\leq C_i\}}$ is a censoring indicator. A common scenario in which this data arise is a clinical trial in which the object of interest is the survival time of patients with some disease. In this scenario, we assess two groups of these patients. The first group receives a treatment (group 1), and the second group receives a placebo (group 0). While for many patients is possible to observe the time of interest -typically death time- for some of these patients this time is not available usually because they leave the study or because the study ends. In the latter case, we just record the last time the patient was seen alive and call this observation right-censored. In this example $g_i$ denotes the groups of the patients, $X_i$ the time we observed the death time of the patient or the time she left the study, and $\Delta_i$ indicates whether we observed the death-time of the patient ($\Delta_i=1$), or the time the patient left the study ($\Delta_i=0$).

For this data, we will be interested on testing whether the distributions that generate the time of interest (death time of the patient, which is not always observed) for each group are the same, that is, our null hypothesis is $H_0:F_0=F_1$. The main challenge for this type of data is censoring, in particular it can happen that  the actual time of interest $\tilde X_i$ is independent of the group $g_i$, but the \emph{observed} time $X_i$ depends on the group $g_i$ (since $X_i$ is also a function of $C_{i}$), that is, we can have $F_0 = F_1$ but $G_0 \neq G_1$, and thus applying a two sample test to the data $X_i$ is not enough.

One of the pillars of survival analysis is the so-called weighted log-rank test, which is a well-known statistic for the two-sample problem in survival analysis. Here, we consider the kernel log-rank test-statistic $\Psi_n$ for the two-sample problem proposed in \cite{fernandez2021reproducing} defined as
\begin{align*}
\Psi_n=\sup_{\omega\in\mathcal{H}:\LH{\omega}=1}S_n(\omega)^2, \quad \text{ where}\quad    S_n(\omega)=\sqrt{\frac{n}{n_0n_1}}\int_0^{\infty}\omega(x)\frac{Y_0(x)Y_1(x)}{Y(x)}\left(\frac{dN_0(x)}{Y_0(x)}-\frac{dN_1(x)}{Y_1(x)}\right),
\end{align*}
here $S_n(\omega)$ is the weighted log-rank estimator. To understand $S_n$ we need to introduce some standard notation used in Survival Analysis. For $\ell \in \{0,1\}$,  $n_{\ell}$ is the sample size of group $\ell$, $N_{\ell}(x)= \sum_{i=1}^n \Delta_i \ind_{\{X_i\leq x,g_i=\ell\}}$ is a counting process (which counts the number of observed events in group $\ell$  up to time $x$), $Y_\ell(x)=\sum_{i=1}^{n}\ind_{\{X_{i}\geq x,g_i=\ell\}}$ counts the number of patients in group $\ell$ that are still in the study by time $x$, $Y(x)=Y_0(x)+Y_1(x)\leq n$ and $\tau_n=\max\{X_1,X_2,\ldots,X_n\}$. Standard computations show that  $\E(Y_0(x)) = n_0 (1-F_0(x))(1-G_0(x))$ and $\E(Y_1(x)) = n_1 (1-F_1(x))(1-G_1(x))$.

Notice that this example is interesting as in this setting $S_n(\omega)$ is not a sum i.i.d. of random variables, but something a bit more complex. Despite this, our results can still be applied to understand the asymptotic distribution of $\Psi_n$ under the null and alternative hypotheses, as all the ingredients needed are already well-known results in Survival Analysis. Additionally, we highlight that compared to the analysis presented \cite{fernandez2021reproducing}, our analysis is much simpler and shorter, and does not require the use of complicated tools such as multiple stochastic integrals, and martingale convergence theorems.

\subsubsection{Analysis under the null hypothesis}

In our analysis we will assume that \Cref{Condi: MMD example} holds. We begin by studying the behaviour of $\Psi_n$ under the null hypothesis (where $F_0 = F_1$). For that define the bilinear form  $\sigma:\mathcal H \times \mathcal H \to \R$ as:

\begin{align*}
\sigma(\omega,\omega') = \int_{0}^{\infty}\omega(x)\omega'(x)\frac{y_0(x)y_1(x)}{\rho_0y_0(x)+\rho_1 y_1(x)} \frac{dF_0(x)}{1-F_0(x)},    
\end{align*}
where $y_\ell(x) = (1-G_\ell(x))(1-F_\ell(x))$ for $\ell\in\{0,1\}$.

\begin{lemma}\label{lemma:kernelLogrank1}
Assume that \Cref{Condi: MMD example} holds. Then for any $\omega_1,\ldots, \omega_{\ell}\in \mathcal H$ we have that under the null distribution
$$(S_{n}(\omega_1),S_{n}(\omega_2),\ldots,S_{n}(\omega_\ell))\overset{\mathcal D}{\to} N(0,\Sigma),$$
where $\Sigma_{ij} = \sigma(\omega_i,\omega_j)$. 

Moreover, it exists a constant $C>0$ such that for every $\omega\in \mathcal H$ we have $\sigma(\omega,\omega)\leq C\int \omega(x)^2 dF_0(x)$ and such that for  all large enough $n$ we have
$$\E(S_{n}(\omega)^2)\leq C\int_0^{\infty} \omega(x)^2dF_0(x).$$
\end{lemma}

The proof of the first part of the previous lemma appears in \cite[Lemma 1]{brendel2014weighted}. The second part follows by noticing that
\begin{align}
    \E(S_{n}(\omega)^2)&=\E\left(\frac{n}{n_0n_1}\int_0^{\infty}\omega(x)^2\frac{Y_0(x)Y_1(x)}{Y(x)}\frac{dF_0(x)}{1-F_0(x)}\right)\leq\frac{n}{n_0n_1}\int_0^{\infty}\omega(x)^2\E\left(Y_0(x)\right)\frac{dF_0(x)}{1-F_0(x)}\nonumber\\
    &=\frac{n}{n_1}\int_0^{\infty}\omega(x)^2(1-G_0(x))dF_0(x)\leq C\int_0^{\infty}\omega(x)^2dF_0(x), \label{eqn:Elogranksquare}
\end{align}
where the first equality is from \cite[Lemma 4.1.2 and its proof]{gill1980censoring}, then we use that $Y_0+Y_1 = Y$, that $\E(Y_0)=n_0(1-G_0)(1-F_0)$, and that $n/n_1$ converges to a constant, so it is bounded for large enough $n$.

\begin{proposition}
Assume that \Cref{Condi: MMD example} holds. Then, under the null hypothesis, $\Psi_n\overset{\mathcal D}{\to}\Psi:=\sum_{i=1}^{\infty}\lambda_i Z_i^2$, where $Z_i$ are i.i.d standard normal variables, and $\lambda_i$ are the eigenvalues of the operator $T_{\sigma}$ associated with $\sigma$.
\end{proposition}
\begin{proof}
We proceed to verify the conditions of \Cref{thm:ConvergenceSumChiSq}.  \Cref{Cond:0bilinear} is established in \Cref{lemma:kernelLogrank1}. For \Cref{Cond:2Tail} consider an orthonormal basis $(\phi_i)_{i\geq 1}$ of $\mathcal H$, then \Cref{lemma:kernelLogrank1} yields that it exists a constant $C>0$ such that $\sigma(\phi_i,\phi_i)\leq C \int\phi_i(x)^2dF(x)$ for all $i\geq 1$, and thus the condition follows directly from \Cref{lemma:twosampleTraceClass}. Finally, for 
\Cref{Cond:1bound}, let $V_i$ be the span of $\phi_1,\ldots, \phi_i$. Then $\|S_n\circ P_{V_i^{\intercal}}\|_{\mathcal H \to \R}^2 = \sum_{k=i+1}^{\infty} S_n(\phi_k)^2$, and thus by the markov inequality and  \Cref{lemma:kernelLogrank1} we have
\begin{align*}
   \Prob(\|S_n\circ P_{V_i^{\intercal}}\|_{\mathcal H \to \R}^2 \geq {\varepsilon})\leq \frac{1}{\varepsilon} \sum_{k=i+1}^{\infty} \E(S_n(\phi_k)^2) \leq \frac{C}{\varepsilon} \sum_{k=i+1}^{\infty}\int \phi_k(x)^2dF_0(x),
\end{align*}
and the last term tends to 0 as $i$ grows by \Cref{lemma:twosampleTraceClass}.
\end{proof}

\subsubsection{Analysis under the alternative hypothesis}
\begin{proposition}
It exists $c^\star\geq 0$ such that under the alternative hypothesis we have $\frac{n}{n_0n_1}\Psi_n\to (c^\star)^2$ a.s.
\end{proposition}

\begin{proof}
We proceed to verify the conditions of \Cref{Thm:alternati} to prove almost sure convergence. It is shown in \cite[Section 7.3]{Flemming91Counting} that
$$\sqrt{\frac{n}{n_0n_1}}S_n(\omega) \overset{a.s.}{\to} c(\omega)= \int_0^{\infty}\omega(x) \frac{y_0(x)y_1(x)}{\rho_0y_0(x)+\rho_1y_1(x)}\left(\frac{dF_0(x)}{1-F_0(x)} -\frac{dF_1(x)}{1-F_1(x)}\right).$$

Following  similar steps as in \cref{eqn:Elogranksquare}, we get that $\sqrt{n/(n_0n_1)}\E(S_n(\omega)^2)\leq C\int \omega(x)^2( dF_0(x)+dF_1(x))$ holds for every $\omega \in \mathcal H$, where $C$ is a constant independent of $\omega$. Now, if we choose a basis $(\phi_i)_{i\geq 1}$ of $\mathcal H$, by \Cref{lemma:twosampleTraceClass} we have 
$$\lim_{u\to \infty}\limsup_{n\to \infty} \sum_{i=u+1}^{\infty} \frac{n}{n_0n_1}S_n(\phi_i)^2 \leq \lim_{u\to \infty} C\int \sum_{i=u+1}^{\infty}\phi_i(x)^2( dF_0(x)+dF_1(x)) = 0. $$

We conclude, by \Cref{Thm:alternati}, that $\frac{n}{n_0n_1}\Psi_n\overset{a.s.}{\to} (c^\star)^2 = \sup_{\unitball}c(\omega)^2.$
\end{proof}

Similar to the standard two-sample problem, we need to ensure that $c^\star>0$. For that we also require the RKHS $\mathcal H$ to be $c_0$-universal (recall its definition from \Cref{sec:alternativetwosample}), and also that the signed sigma-finite measure $\mu$ on $[0,\infty)$ given by
$$\mu(A) =  \frac{y_0(x)y_1(x)}{\rho_0y_0(x)+\rho_1y_1(x)}\left(\frac{dF_0(x)}{1-F_0(x)} -\frac{dF_1(x)}{1-F_1(x)}\right)$$
is different from 0. The last part requires some structural conditions on the censoring distributions $G_0$ and $G_1$ (i.e. the fact that $F_0\neq F_1$ is not enough). This issue is beyond the interest of this work, and we refer the reader to \cite[Section 4.2]{fernandez2021reproducing} for more details.

\subsubsection{Wild boostrap resampling scheme}

Consider $W_i$ as i.i.d. random variables with mean 0 and variance 1, and $\ell \in \{0,1\}$ define the weighted counting processes $N^W_{\ell}(x) = \sum_{i=1}^n W_i\Delta_i \ind_{\{X_i\leq x, g_i=\ell\}}$, and define $S_n^W$ as $S_n$ but replacing $N_{\ell}$ by $N^W_{\ell}$, i.e.
$$S_n^W(\omega)=\sqrt{\frac{n}{n_0n_1}}\int_0^{\tau_n}\omega(x)\frac{Y_0(x)Y_1(x)}{Y(x)}\left(\frac{dN^W_0(x)}{Y_0(x)}-\frac{dN^W_1(x)}{Y_1(x)}\right).$$

Denote $\E_D, \Prob_D$, etc probability conditioned on all data points $(X_i,\Delta_i,g_i)$ so the only source of randomness are the weights $W_i$, and recall the definition of $\overset{\mathcal D_D}{\to}$ from \cref{eqn:defDD}.

For our analysis, let's define
\begin{align*}
    \sigma^W(\omega,\omega') = \int \omega(x)\omega'(x)\frac{y_0(x)y_1(x)}{(\rho_0y_0(x)+\rho_1y_1(x))^2}\left(\rho_1y_1(x)\frac{dF_0(x)}{1-F_0(x)}+\rho_0 y_0(x)\frac{dF_1(x)}{1-F_1(x)}\right),
\end{align*}
and we note that under the null hypothesis $\sigma^W = \sigma$ (as defined in \cref{lemma:kernelLogrank1}).

The following lemma is shown in the proofs of Theorems 5 and 6 of \cite{Ditzhaus2019Wild}.

\begin{lemma}\label{lemma:kernelLogrank2} Suppose that \Cref{Condi: MMD example} holds, and consider the null or the alternative hypothesis. Then for any  $\omega_1,\ldots, \omega_{\ell}\in \mathcal H$  we have that 
$$(S^W_{n}(\omega_1),S^W_{n}(\omega_2),\ldots,S^W_{n}(\omega_\ell))\overset{\mathcal D_D}{\to} N(0,{\Sigma})$$
where $\Sigma_{ij} = \sigma^W(\omega_i,\omega_j)$. Moreover, we have that $\sigma^W(\omega,\omega) \leq \int_0^{\infty}\omega(x)^2 (dF_0(x)+dF_1(x))$.

Furthermore, for any $\omega \in \mathcal H$, it holds that $\E(S_n^W(\omega)) = 0$ and that
\begin{align}\label{eqn:wildVarLogrank}
    \E_D(S_n^W(\omega)^2) = \frac{n}{n_0n_1}\int_0^{\infty} \omega(x)^2\left(\frac{Y_0(x)Y_1(x)}{Y(x)}\right)^2\left(\frac{dN_0(x)}{Y_0(x)^2}+ \frac{dN_1(x)}{Y_1(x)^2}\right).
\end{align}
\end{lemma}
\begin{proposition}
Assume \Cref{Condi: MMD example} holds. Then, $\Psi_n^W \to \Psi^W:= \sum_{i=1}^{\infty} \lambda_i Z_i^2$, where $Z_i$ are i.i.d. standard normal random variables, and $\lambda_i$ are the eigenvalues of the operator $T_{\sigma^W}$ associated with $\sigma^W$.
\end{proposition}
\begin{proof}
We can now verify the conditions of \Cref{thm:ConvergenceSumChiSq} for $\Psi_n^W$. Here we do not need to assume the null or the alternative hypothesis. \Cref{Cond:0bilinear} follows immediately from \Cref{lemma:kernelLogrank2}. \Cref{Cond:2Tail} holds because $\sigma^W(\omega,\omega)\leq \int\omega(x)^2(dF_0(x)+dF_1(x))$ as stated in \cref{eqn:wildVarLogrank}, and thus for a basis $(\phi_i)_{i\geq 1}$ of $\mathcal H$ we have $\sum_i \sigma^W(\phi_i,\phi_i) \leq \sum_i \int\phi_i(x)^2(dF_0(x)+dF_1(x))
$
which converges by \Cref{lemma:twosampleTraceClass}.

Finally, to verify \Cref{Cond:1bound} let $V_i$ be the span of $\phi_1,\ldots, \phi_i$. Define $L(x)$ as $L(x)= Y_0(x)Y_1(x)/Y(x)$, then by the Markov inequality and \cref{eqn:wildVarLogrank}, we have
\begin{align*}
   \Prob(\|S_n P_D{V_i^{\intercal}}\|_{\mathcal H \to \R}^2)& \leq \sum_{k=i+1}^{\infty} \E_D(S_n^W(\phi_k)^2)= \frac{n}{n_0n_1} \int_0^{\infty} \sum_{k=i+1}^{\infty}\phi_k(x)^2 L(x)^2  \left(\frac{dN_0(x)}{Y_0(x)^2}+ \frac{dN_1(x)}{Y_1(x)^2}\right)\\
    &= \frac{n}{n_0n_1} \int_0^{\infty} \LH{P_{V_i^\perp}K_x}^2 L(x)^2  \left(\frac{dN_0(x)}{Y_0(x)^2}+ \frac{dN_1(x)}{Y_1(x)^2}\right).
\end{align*}
We upper bound the last quantity by using that $L(x)^2\leq \min\{Y_0(x)^2,Y_1(x)^2\}$ since $0\leq Y_0+Y_1\leq Y$,  and that $n_0/n$ and $n_1/n$ converge to $\rho_0> 0$ and $\rho_1>0$ respectively, so for large enough $n$ they are bounded above by some constant $C>0$. Then we obtain that
\begin{align*}
    &\frac{n}{n_0n_1} \int_0^{\infty} \LH{P_{V_i^\perp}K_x}^2 L(x)^2  \left(\frac{dN_0(x)}{Y_0(x)^2}+ \frac{dN_1(x)}{Y_1(x)^2}\right)\leq  C\int_0^{\infty} \LH{P_{V_i^\perp}K_x}^2  \left( \frac{dN_0(x)}{n_0}+ \frac{dN_1(x)}{n_1}\right),
\end{align*}
which by the law of the large numbers converges, almost surely, to 
$$\int_0^{\infty} \LH{P_{V_i^\perp}K_x}^2  \rho_0(1-G_0(x))dF_0(x)+\rho_1(1-G_1(x))dF_1(x) \leq \int_0^{\infty} \LH{P_{V_i^\perp}K_x}^2  (dF_0(x)+dF_1(x)). $$
We conclude that 
\begin{align}
  \limsup_{i\to \infty}\limsup_{n\to \infty}  \Prob(\|S_n P_{V_i^{\intercal}}\|_{\mathcal H \to \R}^2) \leq  \limsup_{i\to \infty}\int_0^{\infty} \LH{P_{V_i^\perp}K_x}^2  (dF_0(x)+dF_1(x)) = 0 \text{ }a.s.,
\end{align}
where the last equality is due to \Cref{lemma:twosampleTraceClass}, since $ \LH{P_{V_i^\perp}K_x}^2=\sum_{k=i+1}^\infty\phi_k(x)^2$, and the kernel is bounded.

\end{proof}

To conclude our analysis, since the covariance functions $\sigma$ and $\sigma^W$ coincide under the null hypothesis, we have that $\Psi_n^W$ and $\Psi_n$ have the same limiting distribution under the null. However, under the alternative we have that $\Psi_n\to \infty$ (assuming that $c^\star>0$) and that $\Psi_n^W = O_p(1)$. Then we obtain the same conclusions as in the previous example for the (standard) two-sample problem, in particular that an asymptotically correct testing procedure can be built with the aid of wild bootstrap.

As a final remark, some authors use $\omega(\widehat F(x))$ instead of $\omega(x)$ in the definition of $S_n(\omega)$ (the log-rank statistic) where $\widehat F(x)$ is the Kaplan-Meier estimator using all the data points. This transformation scales the data to $[0,1]$ so we can avoid subpar performance of the testing procedure due to the scale of the data points. Our results still apply (with minor modifications) in this setting.

\section{A new application to conditional independence testing}
We introduce a new test for testing conditional independence based on the kernelisation of the Generalised Covariance Measure recently introduced by \citet{shah2020hardness} and its weighted version studied in \cite{scheidegger2021weighted}.

\subsection{Kernelised  generalised covariance measure }
Consider data points $(X_i,Y_i,Z_i)_{i=1}^n\overset{i.i.d.}{\sim}P$, where $P$ is a probability measure on $\R\times\R\times\R^{d}$, with $d\geq 1$. We are interested on testing whether $X$ and $Y$ are conditionally independent given $Z$. We start by noting that the following decomposition always holds:
\begin{align*}
    X=f(Z)+\epsilon_X(Z),\qquad \text{and} \qquad Y=g(Z)+\epsilon_Y(Z),
\end{align*}
where $f(z)=\E(X|Z=z)$ and $g(z)=\E(Y|Z=z)$.  In order to test the null  hypothesis $H_0:X\perp Y|Z$ against the alternative $H_a:X\not \perp Y|Z$, the following parameter, called the Generalised Covariance Measure $(\GCM)$ was introduced in \cite{shah2020hardness}:
$$\GCM(X,Y;Z) = \E(\epsilon_X(Z)\epsilon_Y(Z)).$$
Under the null hypothesis $\GCM(X,Y;Z) = 0$, so it can be used as a parameter to test the null hypothesis. A weighted generalisation of the $\GCM$, denominated the weighted generalised covariance measure ($\wGCM$) was introduced in \cite{scheidegger2021weighted}. Given a weight function $\omega:\R^{d}\to \R$, the $\wGCM$  is defined as
\begin{align}
\wGCM(X,Y;Z) = \E(\omega(Z)\epsilon_X(Z)\epsilon_Y(Z)).\label{eqn:WGCM}
\end{align}
Again, under the null hypothesis we have that $\wGCM(X,Y;Z) = 0$ for any $\omega \in \mathcal H$. The motivation behind the weighted generalisation of the $\GCM$ is that under some alternatives we may have $\GCM(X,Y;Z) = 0$, but if we choose an appropriate weight we will get $\wGCM(X,Y;Z)\neq 0$, and so the weighted version should be more robust.

In order to use the $\GCM$ and the $\wGCM$ in practice, we need to estimate $\epsilon_X$ and $\epsilon_Y$ from the data. For that, we need to estimate the conditional expectation of $X$ and $Y$ given $Z$, which can be done by a regression estimator. Denote by $\hat f$ and $\hat g$ the regression estimators of $\E(X|Z)$ and $\E(Y|Z)$, respectively. Here $\hat f$ is estimated using $(X_i,Z_i)_{i=1}^n$ whereas $\hat g$ is estimated using $(Y_i,Z_i)_{i=1}^n$. Then define
\begin{align}
    \widehat \epsilon_{X_i}(Z_i) = X_i - \widehat f(Z_i) \quad \text{ and }\quad \widehat \epsilon_{Y_i}(Z_i) = Y_i - \widehat g(Y_i)
\end{align}
and note we can estimate the $\wGCM(X,Y;Z)$ by $\frac{1}{n} \sum_{i=1}^n \omega(Z_i) \widehat\epsilon_{X_i}(Z_i) \widehat\epsilon_{Y_i}(Z_i)$, which should be close to 0 under the null hypothesis. For our developments, it is more convenient to re-scale the previous estimate by $\sqrt{n}$ to obtain the test-statistic

$$S_n(\omega) = \frac{1}{\sqrt n} \sum_{i=1}^n \omega(Z_i)\widehat \epsilon_{X_i}(Z_i)\widehat \epsilon_{Y_i}(Z_i).$$

\subsubsection{A test based on the Kernelised GCM}
As expected, the $\wGCM$ requires the user to input a weight function that needs to be chosen carefully. Following the spirit of this work, we consider the kernelisation of the $\wGCM$, thus our test statistic is
\begin{align*}
    \Psi_n&=\sup_{\omega\in\mathcal{H}:\LH{\omega}=1} S_n(\omega)^2,\quad\text{where}\quad    S_n(\omega)=\frac{1}{\sqrt{n}}\sum_{i=1}^n \omega(Z_i)\widehat \epsilon_{X_i}(Z_i)\widehat \epsilon_{Y_i}(Z_i).
\end{align*}
To analyse the kernelised $\GCM$ we need some conditions on the regression estimators $\hat f$ and $\hat g$, in order to ensure that the estimation is good enough, as well as some other regularity conditions. 

\begin{condition}\label{Cond:condIndep}
Consider the following quantities:
\begin{align*}
    A_f=\frac{1}{n}\sum_{i=1}^n(f(Z_i)-\widehat{f}(Z_i))^2&,\qquad u_f(z,y)=\E(\epsilon_X(Z)^2|Z=z,Y=y)\\
    A_g=\frac{1}{n}\sum_{i=1}^n(g(Z_i)-\widehat{g}(Z_i))^2&,\qquad v_g(z,x)=\E(\epsilon_Y(Z)^2|Z=z,X=x).
\end{align*}
We assume the following conditions hold. 
\begin{itemize}
    \item[i.] $A_f=o_p(n^{-1/2})$ and $A_g=o_p(n^{-1/2})$, and
    \item[ii.] $u_f(z,y)$ and $v_g(z,x)$ are uniformly bounded
    \item[iii.] $0<\E(\epsilon_X^2\epsilon_Y^2)$. 
    \item[iv.] There exists a constant $C>0$ such that $|K(z,z')|\leq C$ for all $z,z'\in\R^{d}$
\end{itemize}
\end{condition}

\begin{remark}
 The above conditions are slightly stronger than the corresponding conditions of \cite{scheidegger2021weighted}, but they allow us to avoid splitting the data as  in \cite{scheidegger2021weighted} (e.g. use half of the data to estimate $f$ and $g$, and the other half in the testing procedure). Our conditions now require that the conditional variances $u_P(z,y)$ and $v_P(z,x)$ are uniformly bounded, which implies $\E(\epsilon_X^2\epsilon_Y^2)<\infty$. This condition is not at all restrictive and in many setting it is assumed that the variance of the regression errors is bounded, and moreover, it can easily be relaxed at the price of having less clear statements and longer proofs.
\end{remark}
\begin{remark}
 Note that under the null hypothesis, $u_p$ and $v_p$ are only functions of $z\in\R^{d}$.
\end{remark}

We proceed to enunciate the main results that ensure that the proposed test-statistic leads to an asymptotically correct test. We defer the proofs of those results to \Cref{sec:deferedProofs} after the experiments, so we keep the focus on the results of this new test, rather than on the technical details.

The next theorem gives the limit distribution of $\Psi_n$ under the null hypothesis.
\begin{theorem}\label{thm:conditest}
Suppose that  \Cref{Cond:condIndep} holds. Then, under the null hypothesis it holds
\begin{align*}
\Psi_n&=\sup_{\unitball} S_n(\omega)^2\overset{\dist}{\to}\Psi:=\sum_{i=1}^\infty\lambda_i\Q_i^2,
\end{align*}
as $n$ grows to infinity, where $\Q_1,\Q_2,\ldots$ are i.i.d. standard normal random variables, and $\lambda_1,\lambda_2,\ldots$ are the eigenvalues associated to the operator $T_\sigma:\mathcal{H}\to \mathcal{H}$ associated to the covariance 
\begin{align}
    \sigma(\omega, \omega') = \int_{\R^{d}} \omega(z)\omega'(z) \E(\epsilon_X(Z)^2|Z)\E(\epsilon_Y(Z)^2|Z)dF_Z(z).\label{eqn:Tsigma2}
\end{align}
\end{theorem}

Under the alternative hypothesis we have the following asymptotic result.

\begin{theorem}\label{thm:ALteCLI}
Suppose that  \Cref{Cond:condIndep} holds. Then, under the alternative hypothesis it exists a constant $c^\star\in \R$ such that $\frac{1}{n}\Psi_n\overset{\Prob}{\to} (c^\star)^2$. Moreover, if it exists a measurable $A\subseteq \R^{d}$ such that  $\E(\epsilon_X(Z)\epsilon_Y(Z)|Z) \neq 0$ for $Z \in A$, and $\mathcal H$ is $c_0$-universal (see \Cref{sec:alternativetwosample}) then $c^\star \neq 0$.
\end{theorem}

The above shows that if $c^\star \neq 0$, then $\Psi_n\to \infty$ under the alternative, so asymptotically we should reject the null if since $\Psi_n=O_p(1)$ under the null. Nevertheless, we do not know the theoretical distribution under the null, so we shall use a resampling scheme to construct a rejection region. 

Define the wild bootstrap version of $\Psi_n$ as 
\begin{align*}
    \Psi_n^W&=\sup_{\omega\in\mathcal{H}:\LH{\omega}=1} S_n^W(\omega)^2,\quad\text{ where }\quad    S_n^W(\omega)=\frac{1}{\sqrt{n}}\sum_{i=1}^n W_i\widehat \epsilon_{X_i}(Z_i)\widehat \epsilon_{Y_i}(Z_i)\omega(Z_i),
\end{align*}
where $W_1,\ldots,W_n$ are i.i.d. random variables with mean 0 and variance 1. As usual, as we are conditioning on the data points, we write $\Prob_D,\E_D,$ etc, as defined in \Cref{sec:twoSampleWB}.

\begin{theorem}\label{thm:Wild} Suppose that  \Cref{Cond:condIndep} holds. Then, 
$\Psi_n^W\overset{\dist_D}{\to}\Psi^W:=\sum_{i=1}^\infty\lambda_i\Q_i^2$, where $\Q_1,\Q_2,\ldots$ are i.i.d. standard normal random variables, and $\lambda_1,\lambda_2,\ldots$ are the eigenvalues of the operator $T_{\sigma^W}:\mathcal{H\to\mathcal{H}}$ associated with the covariance
\begin{align}
     \sigma^W(\omega,\omega')&=\int_{\R^{d}}\omega(z)\omega'(z)\E(\epsilon_X(Z)^2\epsilon_Y(Z)^2|Z=z)dF_Z(z).\label{eqn:gamma}
     \end{align}
 \end{theorem}
 We first remark that the theorem above holds under the null and under the alternative hypotheses. Now, note that, on the one hand, under the null we have that $\sigma^W = \sigma$ as defined in \cref{eqn:Tsigma2}, so the limiting distribution  $\Psi_n$ and $\Psi_n^W$ are the same, meaning that under the null hypothesis we can use $\Psi_n^W$ to resample from $\Psi_n$ to approximate its rejection region with asymptotic guarantees (by resampling from $\Psi_n^W$ as much as needed), so our test asymptotically reaches the desired level. On the other hand, note that by combining \Cref{thm:ALteCLI} and \Cref{thm:Wild} we obtain that under the alternative hypothesis $\Psi_n^W = O_p(1)$ while $\Psi_n \to \infty$ (under the conditions of \Cref{thm:ALteCLI}) which means that under the alternative our test rejects the alternative with probability tending to 1. Also, note that \Cref{Cond:condIndep}.iii ensures that $\sigma$ and $\sigma^W$ are different from zero, so the limiting distributions are not trivial.

\subsubsection{Experiments}\label{sec:ExperimentsKGCM}

In this section, we study the behaviour of the Kernelised $\GCM$, from now on $\KGCM$, in two simulated data sets in order to evaluate its performance in practice. In our experiment we implement three versions of the kernelised $\GCM$, from now on $\KGCM$, by choosing three different RKHS's (or rather three different kernels). We also implemented the test based on the $\wGCM$ and the $\GCM$ for comparison purposes.

Regarding implementations details, for the $\KGCM$, we choose the kernels as the squared exponential kernel $K_\ell(z,z'):\R^{d}\times \R^{d}\to \R$, which is given by $K_\ell(z,z')=\exp\{-\frac{1}{\ell^2}\|z-z'\|^2\}$, where $\ell^2\in \R$ is the length-scale (or bandwidth) parameter. The length-scale parameter $\ell$ controls the fluctuations of the functions of $\mathcal H$. A larger length-scale parameter is associated with flatter curves, whilst a smaller one is associated with functions with more fluctuations. Thus, a smaller length-scale parameter should be preferred for problems involving non-linear structures. A known heuristic to choose the length-scale parameter is the \textit{median heuristic} which chooses $\ell^2$ as the median of all the pairwise differences $\|Z_i-Z_j\|^2$ for $i,j\in\{1,\ldots,n\}$. In our experiments we implement three versions of the $\KGCM$, which we name as \textbf{$\KGCM$-1}, \textbf{$\KGCM$-2} and \textbf{$\KGCM$-3}, in which we use the length-scales: $\{0.1,1,\textit{median heuristic}\}$, respectively. While we do not pursue the goal of finding the best length-scale in our experiments, we remark that the problem of choosing an appropriate length-scale is currently a very active research topic in statistics and machine learning, and we refer to \cite{Albert2022Adaptive,schrab2021mmd,Schrab2022KSD,schrab2022efficient} for very recent results in the area. To find rejection regions we use wild bootstrap. In all cases we use $M=1000$ independent wild bootstrap samples (i.e. we sampled 1000 times the weights $(W_i)_{i=1}^n$), which are used to construct the region. We choose the weights as Rademacher random variables. In our experiments we choose the level of the test as $\alpha = 0.05$. Since the test-statistic $\Psi_n^2$ is non-negative, the rejection for the $\KGCM$ is chosen as $(q_{1-\alpha}^M,\infty)$ where $q_{1-\alpha}^M$ is the value in position $(1-\alpha)M$ of the $M$ bootstrap samples (when sorted in increasing order).
  
In order to compare our methods, we implement the $\GCM$ and the $\wGCM$. For the $\wGCM$, we use fixed weight functions, and we refer to \cite{scheidegger2021weighted} for details on the implementation. Our first experiment considers $d = 1$, and in such case we consider the weight function $\omega(z)=\text{sign}(z)$ (as it was used in the experiments of \cite{scheidegger2021weighted}). Our second experiment has $d>1$, and thus we combine $d+1$ weight functions $\omega_0,\omega_1,\ldots,\omega_{d}$ in the $\wGCM$. Following \cite{scheidegger2021weighted} the functions are chosen as  $\omega_0(\boldsymbol{z})=1$ and $\omega_i(\boldsymbol{z})=\text{sign}(z_i)$ for any $i\in\{1,\ldots,d\}$. The implementation of the $\GCM$ is rather straightforward because it is the $\wGCM$ with weight function $\omega = 1$.

Recall that the $\GCM$ and its variations require the estimation of the conditional means $\E(X|Z)$ and $\E(Y|Z)$. For this task we use polynomial regression as for our simulated data it is enough to have a good estimate. This allows us to focus on the testing part of the problem, rather than the estimation part of the problem. However, we highlight that $\GCM$, $\wGCM$ and the $\KGCM$ rely on selecting a good regression procedure satisfying \Cref{Cond:condIndep}, and thus, for complex datasets we would require more sophisticated regression methods to perform well. 
We proceed to describe our data sets, and the obtained results.

\begin{itemize}
    \item \textbf{Data 1:} Let $U_1\sim N(0,1)$ and $U_2\sim N(0,1)$ be independent. Given a parameter $\gamma\in[0,1]$, we generate data as $Z\sim N(0,1)$, $X=Z+U_1\sin(5Z)$ and $Y=Z^2+\gamma U_1+(1-\gamma)U_2$.
\end{itemize}

In this experiment we vary the parameter $\gamma$, so we compute rejection rates for each of them (in a grid). Our experiments consider $n=100$ data points, and to estimate the rejection rate we repeat 1000 times the experiment. The results of this experiment are shown in \Cref{fig:my_label}.

On the one hand, it is not difficult to see that if $\gamma=0$, then the null hypothesis holds. Thus, in this case, the rejection rate should not be greater than the level of the test given by  $\alpha=0.05$. In \Cref{fig:my_label} (bottom right) we observe that all tests show a rejection rate close to $\alpha=0.05$ (dashed black line), which shows a correct Type-I error. On the other hand, when $\gamma$ grows, the conditional dependence of $X$ and $Y$ given $Z$ starts to be more noticeable. Indeed, we would expect that the rejection rate (power) starts growing as $\gamma$ approaches 1 for all tests. Note however that this will not happen for the $\GCM$ as for any value of $\gamma\in[0,1]$, we have that
\begin{align*}
    \E(\epsilon_X(Z)\epsilon_Y(Z))&=\E((X-\E(X|Z))(Y-\E(Y|Z)))=\gamma\E(\sin(5Z))=0,
\end{align*}
because $\E(X|Z)=Z$ and $\E(Y|Z)=Z^2$. 
\begin{figure}
    \centering
    \includegraphics[scale=0.35]{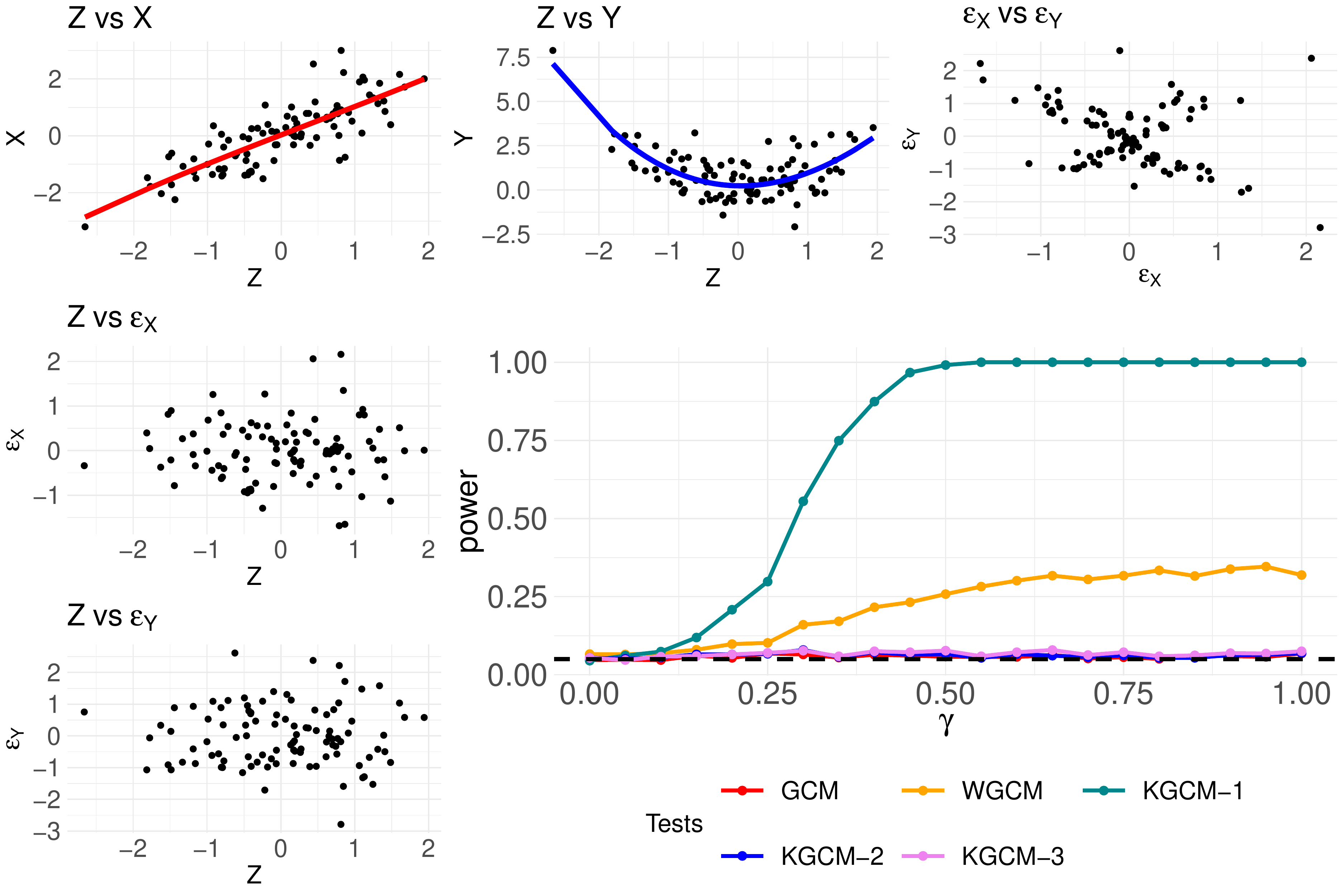}
    \caption{Scatter plots for $\gamma=1$. On the bottom-right corner, the picture shows the power of the different test for different values of $\gamma$. The null hypothesis is recovered only for $\gamma = 0$.}
    \label{fig:my_label}
\end{figure}

As a consequence of the previous result, we expect that the $\GCM$ fails to reject the null hypothesis when the null is false, i.e. when $\gamma>0$. This behaviour can be observed in \Cref{fig:my_label}. 

Our experiments show promising results for the \textbf{$\KGCM$-1} test, which uses a length-scale parameter of $\ell^2=0.1$. This good result can be explained through the fact that a smaller length-scale parameter is associated with functions with more fluctuations and thus it can be a good candidate as our data is generated by  the function $z\to\sin(5z)$. Finally, we observe that the $\wGCM$ is able to detect some dependence, but the results are not optimal.

\begin{itemize}
    \item \textbf{Data 2:} Let $d\geq 2$. We generate data as follows    \begin{align*}
        \boldsymbol{Z}=(Z_1,\ldots,Z_{d})\sim N(0,I_{d}),\quad X=Z_1+\frac{1}{\sqrt{{d}}}\sum_{i=1}^{d} U_iZ_i,\quad\text{and}\quad Y=Z_2+\frac{1}{\sqrt{{d}}}\sum_{i=1}^{d} U_i,
    \end{align*}
where $\boldsymbol{U}=(U_1,\ldots,U_{d})\sim N(0,I_{d})$ is independent of $\boldsymbol{Z}$, and $I_d$ is the $d\times d$ identity matrix.
\end{itemize}
In this experiment we now consider a multivariate $\boldsymbol{Z}$ having $d$ dimensions (we use bold letters to remark the fact that we have a vector in $\R^{d}$). The goal in this experiment is to test whether $X$ and $Y$ are independent given the random vector $\boldsymbol{Z}$. Note that in this case, $X$ is not independent of $Y$ given $\boldsymbol{Z}$ since both $X$ and $Y$ depend on the vector $\boldsymbol{U}$. Also, observe that $\E(X|\boldsymbol{Z})=Z_1$ and $\E(Y|\boldsymbol{Z})=Z_2$, from which it can be easily deduced that $\E(\epsilon_X\epsilon_Y)=0$. Lastly, observe that by the Central Limit Theorem, it holds that
\begin{align*}
\epsilon_X=\frac{1}{\sqrt{d}}\sum_{i=1}^{d}U_iZ_i\overset{\dist}{\to}N(0,1),\quad\text{and}\quad \epsilon_Y=\frac{1}{\sqrt{d}}\sum_{i=1}^{d}U_i\overset{\dist}{\to}N(0,1),
\end{align*}
when $d$ grows to infinity. Then, since $\E(\epsilon_X(\boldsymbol{Z})\epsilon_Y(\boldsymbol{Z}))=0$, we can deduce that actually $(\epsilon_X,\epsilon_Y)$ converges in distribution to a pair of independent standard normal random variables. Thus, we expect to observe loss of power for all our tests as the parameter $d$ starts growing.

The results obtained by the implemented tests are shown in \Cref{fig:my_label2}. As expected, \Cref{fig:my_label2} shows how the power for all tests decreases as the dimension $d$ increases. Also, we can observe that the $\GCM$ fails to reject the alternative for any $d\geq 2$, which is justified by the fact that $\E(\epsilon_X(\boldsymbol{Z})\epsilon_Y(\boldsymbol{Z}))=0$ for any $d\geq 2$. Finally, notice that the best results are attained by the \textbf{$\KGCM$-3} which uses a length-scale parameter computed using the median heuristic.
\begin{figure}[H]
    \centering
    \includegraphics[scale=0.35]{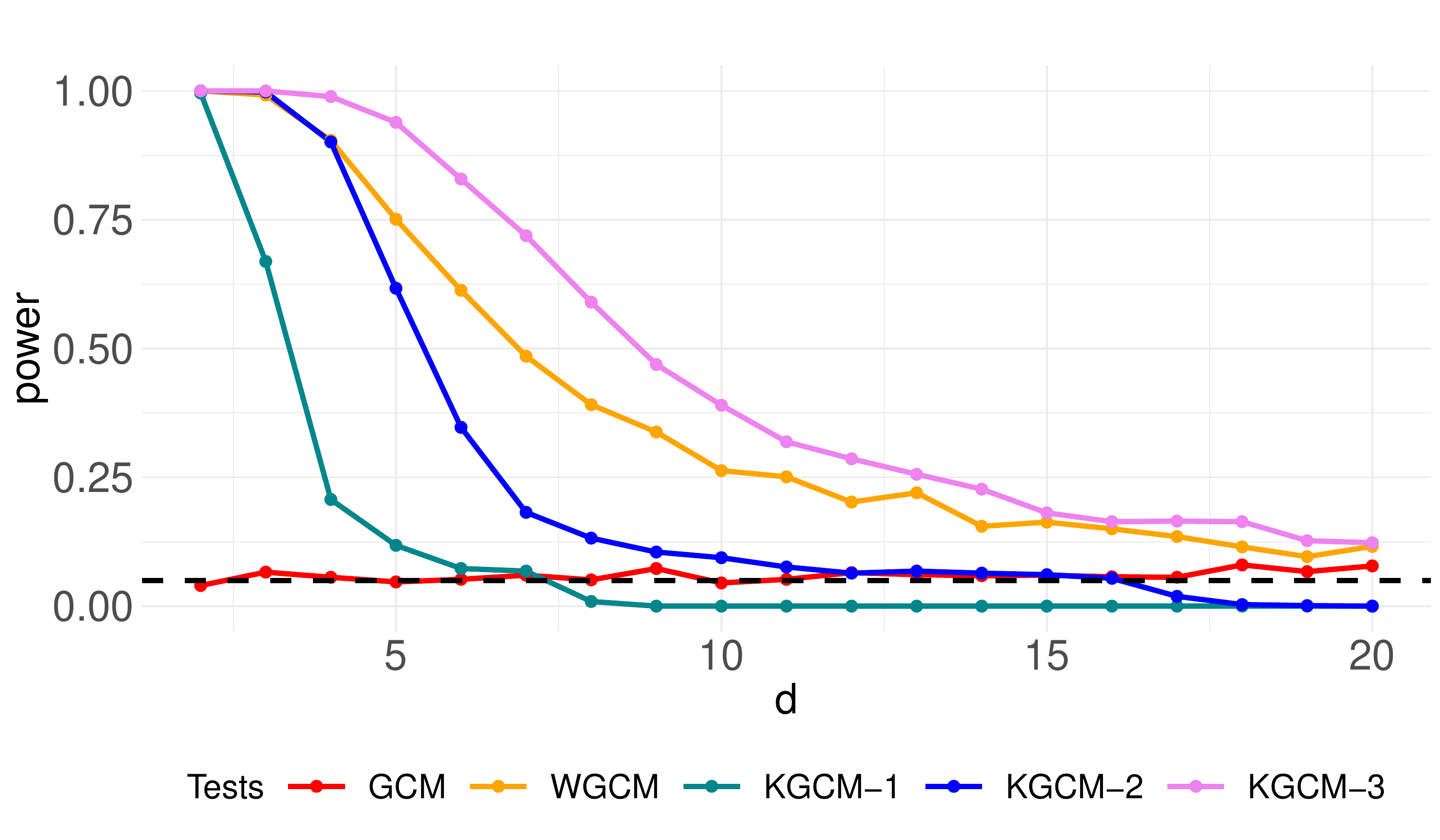}
    \caption{Test power attained by the tests for different values of the dimension $d$. The null hypothesis is recovered when $d\to\infty$.}
    \label{fig:my_label2}
\end{figure}

\subsubsection{Deferred proofs}\label{sec:deferedProofs}
We start with a general approximation result that holds under the null and the alternative hypotheses.

\begin{lemma}\label{Lemma:approx}
Suppose that \Cref{Cond:condIndep} holds. Then, for any $\omega$ in the unit ball of $\mathcal{H}$, it holds
\begin{align*}
S_n(\omega)&=\widetilde{S}_n(\omega)+o_p(1),\qquad\text{where}\qquad\widetilde{S}_n(\omega)=\frac{1}{\sqrt{n}}\sum_{i=1}^n \epsilon_{X_i}(Z_i)\epsilon_{Y_i}(Z_i)\omega(Z_i),
\end{align*}
where the $o_p(1)$ term does not depend on $\omega$.
\end{lemma}
The proof of \Cref{Lemma:approx} follows exactly the same steps of the proof of Theorem 6 of \cite{shah2020hardness}, with very minor modifications. Hence, the proof is omitted from the main text, but included in \Cref{sec:aux} for completeness.

\begin{proof}[{Proof of \Cref{thm:conditest}}]
Consider $\omega$ in the unit ball of $\mathcal{H}$. Then, under \Cref{Cond:condIndep}, \Cref{Lemma:approx} deduces $S_n(\omega)=\widetilde{S}_n(\omega)+o_p(1)$, where the $o_p(1)$ term does not depend on $\omega$. Thus, by Slutsky's theorem, $\Psi_n$ has the same limiting distribution as $\widetilde \Psi_n$, defined as
\begin{align*}
   \widetilde{\Psi}_n=\sup_{\omega\in\mathcal{H}:\LH{\omega}=1}\widetilde{S}_n(\omega)^2.
\end{align*}
We now focus on finding the limit distributing of $\widetilde \Psi_n$ under the null. For that, we will apply \Cref{thm:ConvergenceSumChiSq}, thus we proceed to verify \Cref{Cond:0bilinear,Cond:1bound,Cond:2Tail}. We begin by checking \Cref{Cond:0bilinear}. Note that $\widetilde{S}_n(\omega)$ is a sum of i.i.d. random variables. Thus, by the Central Limit Theorem, we obtain $\boldsymbol{\widetilde{S}_n}=(\widetilde{S}_n(\omega_1),\widetilde{S}_n(\omega_2),\ldots,\widetilde{S}_n(\omega_\ell))\overset{\dist}{\to}N_\ell(0,\boldsymbol{\Sigma})$ as $n\to\infty$, where $\boldsymbol{\Sigma}_{i,j} = \E\left(\omega_i(Z)\omega_j(Z)\epsilon_X(Z)^2\epsilon_Y(Z)^2\right)$ which is equal to $\sigma(\omega_i,\omega_j)$ as defined in \cref{eqn:Tsigma2}.

Note that the boundedness of the first and second moment required by the CLT follow from \Cref{Cond:condIndep}.

To check \Cref{Cond:2Tail}, consider a basis $(\phi_i)_{i\geq 1}$ of $\mathcal H$. Recall that $\sigma(\phi_k,\phi_k) = \int_{R^{d}} \phi_k(z)^2 \E(\epsilon_X^2|Z=z)\E(\epsilon_Y^2|Z=z)dF_Z(z)$, then $\sigma(\phi_k,\phi_k)\leq C\int_{\R^{d}} \phi_k(z)^2 dF_Z(z)$ for some constant $C>0$ since $\E(\epsilon_X^2|Z=z)$ and $\E(\epsilon_Y^2|Z=z)$ are uniformly bounded by \Cref{Cond:condIndep}. Therefore, \Cref{Cond:2Tail} follows from \Cref{lemma:twosampleTraceClass}. 

Finally, we verify \Cref{Cond:1bound}. Let $V_i$ be the span of $\phi_1,\ldots, \phi_i$, and recall that  $\|\widetilde{S}_nP_{V_i^\perp}\|^2_{\mathcal{H}\to\mathbb{R}}=\sum_{k=i+1}^\infty\widetilde{S}_n(\phi_k)^2$, then for any $\varepsilon>0$, we have
\begin{align}
\Prob\left(\|\widetilde{S}_nP_{V_i^\perp}\|^2_{\mathcal{H}\to\mathbb{R}}> \varepsilon\right)
&\leq\frac{1}{\varepsilon}\E\left(\sum_{k=i+1}^\infty\widetilde{S}_n(\phi_k)^2\right)=\frac{1}{\varepsilon}\sum_{k=i+1}^\infty\sigma(\phi_k,\phi_k).
\end{align}
The latter does not depend on $n$, and tends to $0$ when $i$ grows to infinity because \Cref{Cond:2Tail} holds, concluding that $\lim_{i\to \infty}\limsup_{n\to \infty}\Prob(\|\widetilde{S}_nP_{V_i^\perp}\|^2_{\mathcal{H}\to\mathbb{R}}> \varepsilon) = 0$ as desired.

By \Cref{thm:ConvergenceSumChiSq}, $\widetilde{\Psi}_n\to\sum_{i=1}^\infty\lambda_i\Q_i^2$, where $\Q_1,\Q_2,\ldots$ are i.i.d. standard normal random variables and $\lambda_1,\lambda_2,\ldots$ are the eigenvalues of the trace-class operator $T_\sigma:\mathcal{H}\to\mathcal{H}$ associated with $\sigma$.
\end{proof}

\begin{proof}[{Proof of \Cref{thm:ALteCLI}}]
We will start by proving that $\frac{1}{n}\Psi_n\overset{\Prob}{\to} (c^\star)^2$, where $c^\star\in \R$ will be identified later. Let $\widetilde S_n(\omega)$ as in \Cref{Lemma:approx}, and let $\widetilde \Psi_n = \sup_{\unitball}\widetilde S_n(\omega)^2$. By \Cref{Lemma:approx}, we have that $\frac{1}{n} \Psi_n = \frac{1}{n}\widetilde \Psi_n+o_p(1)$, so it is enough to show the result for $\frac{1}n\widetilde \Psi_n$. We will use \Cref{Thm:alternati} for such a task. 

Let $(X,Y,Z)\sim P$, then note that 
\begin{align*}
    \frac{1}{\sqrt{n}}\widetilde S_n(\omega)\overset{\Prob}{\to} c(\omega):=  \E(\epsilon_X(Z)\epsilon_Y(Z)\omega(Z))
\end{align*}
as $n$ grows to infinity. Note that the $c^\star := \sup_{\unitball}c(\omega)<\infty$ is bounded due to \Cref{Cond:condIndep}.ii. and  \Cref{Cond:condIndep}.iii.

We just need to show that for a base $(\phi_i)_{i \geq 1}$ of $\mathcal H$ we have 
\begin{align}
\lim_{i\to\infty} \limsup_{n\to \infty}\Prob\left(\sum_{k=i+1}^{\infty} \frac{1}{\sqrt n} \widetilde S_n(\phi_k)^2\geq \varepsilon\right) = 0.\label{eqn:random39318g1}    
\end{align}
Consider a base $(\phi_i)_{i\geq 1}$, then, $$\frac{1}{\sqrt n}\E(\widetilde S_n(\phi_i)^2) \leq \E\left(\phi_i(Z)^2 \epsilon_X(Z)^2\epsilon_Y(Z)^2\right) \leq C\E(\phi_i(Z)^2)$$
where the first inequality follow from straightforward computations for sums of i.i.d. 
random variables, and the second because \Cref{Cond:condIndep}.ii. Then, \eqref{eqn:random39318g1} follows from \Cref{lemma:twosampleTraceClass}. We conclude that \Cref{Thm:alternati} yields $    \frac{1}{n}\widetilde\Psi_n\overset{\Prob}{\to} (c^\star)^2= \sup_{\unitball} c(\omega)^2.$

Let's assume now that $\E(\epsilon_X(Z)\epsilon_Y(Z)|Z)\neq 0$ for $Z \in A$ with $\Prob(Z\in A)>0$, and consider the measure $\nu$ on $\R^{d}$ given by
$$\nu(B) = \E(\ind_{B}(Z)\epsilon_X(Z)\epsilon_Y(Z)),$$
for any Borel measurable set $B$ of $\R^{d}$. Note that by the previous assumption, $\nu$ is not the zero measure. Moreover, observe that $c(\omega) = \int \omega(z)\nu(dz)$. Then, since $\mathcal H$ is $c_0$-universal, the fact that $\nu$ is not the zero-measure implies that $c^\star>0$.
\end{proof}

\begin{proof}[{Proof of \Cref{thm:Wild}}]
Under \Cref{Cond:condIndep}, by using similar arguments as the ones used in the proof of \Cref{Lemma:approx}, we can prove that $S_n^W(\omega) = \widetilde S_n^W(\omega) + o_p(1)$, where $o_p(1)$ is independent of $\omega$ (since the proof is quite similar and does not add anything new to our analysis, we omit it). Therefore, define $\widetilde \Psi_n^W$ as
\begin{align*}
    \widetilde \Psi_n^W=\sup_{\unitball}\widetilde{S}_n^W(\omega).
\end{align*}
Then, by Slutsky's theorem, the limiting distribution of $\Psi_n^W$ and $\widetilde \Psi_n^W$ are the same (if one of them exist). We proceed to show the existence of a random variable $\Psi^W$ such that  $\widetilde \Psi_n^W\overset{\mathcal D_D}{\to} \Psi^W$ via \Cref{thm:ConvergenceSumChiSq} (recall the definition of $\overset{\mathcal D_D}{\to}$ from \cref{eqn:defDD}). For that we verify  \Cref{Cond:0bilinear,Cond:1bound,Cond:2Tail} conditioned on the data points. To check \Cref{Cond:0bilinear}, note that conditioned on the data, the test-statistic $\widetilde{S}_n^W(\omega)$ is just a sum of independent random variables where
\begin{align*}
    \E_D\left(\widetilde{S}_n^W(\omega)\right)=0\quad\text{and}\quad \Cov_D\left(\widetilde{S}_n^W(\omega),\widetilde{S}_n^W(\omega')\right)=\frac{1}{n}\sum_{i=1}^n\omega(Z_i)\omega'(Z_i)\epsilon_{X_i}(Z_i)^2\epsilon_{Y_i}(Z_i)^2.
\end{align*}
Then, by using the central limit theorem (e.g. Linderberg CLT) we obtain $\boldsymbol{\widetilde{S}}_n^W=(\widetilde{S}_n^W(\omega_1),\widetilde{S}_n^W(\omega_2),\ldots,\widetilde{S}_n^W(\omega_\ell))^\intercal\overset{\dist_D}{\to}N_\ell(0,\boldsymbol{\Sigma}^W)$, where
\begin{align*}
    \boldsymbol{\Sigma}_{i,j}^W = \Cov_D\left(\widetilde{S}_n^W(\omega_i),\widetilde{S}_n^W(\omega_j)\right) \overset{a.s.}{\to} \sigma^W(\omega_i,\omega_j),
\end{align*}
with $\sigma^W$ as defined in \cref{eqn:gamma}.

To check \Cref{Cond:2Tail,Cond:1bound}, choose a basis $(\phi_i)_{i\geq 1}$ of $\mathcal H$. For \Cref{Cond:2Tail} observe that  exists a constant $C>0$ such that $\sigma^W(\phi_i,\phi_i) = \int_{\R^{d}} \phi_i(z)^2\E(\epsilon_X^2\epsilon_Y^2|Z=z)dF_Z(z)\leq C\int_{\R^{d}} \phi_i(z)^2dF_Z(z)$ due to \Cref{Cond:condIndep}.ii. Then $\sum_{i} \sigma^W(\phi_i,\phi_i)\leq C\int \phi_i(z)^2dF_Z(z)$ which is finite by \Cref{lemma:twosampleTraceClass}, yielding \Cref{Cond:2Tail}. 

To verify \Cref{Cond:1bound}, we use that  $\sum_{k=i+1}^{\infty} \phi_k(x)^2 =\LH{P_{V_i^\perp}K_{x}}^2$ to get
\begin{align*}
\Prob_D\left(\|\widetilde{S}_n^WP_{V_i^\perp}\|^2_{\mathcal{H}\to\mathbb{R}}\geq \varepsilon\right)\leq\frac{1}{\varepsilon}\E_D\left(\sum_{k=i+1}^\infty\widetilde{S}_n^W(\phi_k)^2\right)&=\frac{1}{\varepsilon}\sum_{k=i+1}^\infty\left(\frac{1}{n}\sum_{j=1}^n\phi_k(Z_j)^2\epsilon_{X_j}(Z_j)^2\epsilon_{Y_j}(Z_j)^2\right)\\
&=\frac{\varepsilon^{-1}}{n}\sum_{j=1}^n\LH{P_{V_i^\perp}K_{Z_j}}^2\epsilon_{X_j}(Z_j)^2\epsilon_{Y_j}(Z_j)^2.
\end{align*}
Then, since all triples $(X_i,Y_i,Z_i)$ are independent, the law of large numbers yields
\begin{align}
    \limsup_{n\to\infty}\Prob_D(\|\widetilde{S}_n^WP_{V_i^\perp}\|^2_{\mathcal{H}\to\mathbb{R}}\geq \varepsilon)\leq \varepsilon^{-1}\int_{\R^{d}}\LH{P_{V_i^\perp}K_{z}}^2\E(\epsilon_X^2\epsilon_Y^2|Z=z)dF_Z(z).\label{eqqq}
\end{align}
Note that by \Cref{Cond:condIndep} we have $\E(\epsilon_X^2\epsilon_Y^2|Z=z)\leq C$ for some $C>0$, then, since $\LH{P_{V_i^\perp}K_{x}}^2=\sum_{k=i+1}^{\infty} \phi_k(x)^2 $,  \Cref{lemma:twosampleTraceClass} yields  that the right-hand side of  \cref{eqqq} tends to $0$ as $i$ grows to infinity.

 Since the conditions of \Cref{thm:ConvergenceSumChiSq} have been verified, we conclude that  $\widetilde{\Psi}_n^W\overset{\dist}{\to} \Psi^W:=\sum_{i=1}^\infty\lambda_i \Q_i^2$, where $\xi_1,\xi_2,\ldots$ are i.i.d. standard normal random variables and $\lambda_1,\lambda_2,\ldots$ are the eigenvalues associated to $T_{\sigma^W}$. Since we have proven that $\widetilde{\Psi}_n$ converges in distribution, then proof is finished. 
\end{proof}

\section{Conclusion}
We have introduced new tools to analyse the asymptotic behaviour of kernel-based tests. These tools give us necessary and sufficient conditions to extend asymptotic results for standard weighted test-statistics, to kernelised test-statistics, making the analysis of the kernel tests much simpler, cleaner, and shorter. The latter is a direct consequence of the fact that our analysis is carried out directly on random functionals on the Hilbert space, avoiding the intricate expansions that usually appear in the literature of kernel tests. 

To show the wide range of application of our results, we analysed two already known testing procedures, and we exhibit very short proofs of already known results via using our techniques. Additionally, we develop a new kernel-test for conditional independence (testing whether $X$ and $Y$ are independent given $Z$). This test was obtained as the kernelisation of the recently introduced generalised covariance measure. For this test, we present an asymptotic analysis using our developments. To study the practical behaviour of the new test, we perform experiments in two simulated data sets, showing that the kernelised test performs better than the generalised covariance measure and its weighted generalisation. We leave as future work a more detailed study of this new testing procedure, especially in the setting where the dimension of $X$ and $Y$ is greater than 1.
\paragraph{Acknowledgements}
The authors would like to thank Rolando Rebolledo for his comments, suggestions, and stimulating discussions.
T. Fern\'andez was supported by ANID FONDECYT grant No 11221143 and N. Rivera was supported by ANID FONDECYT grant No 3210805.
\bibliography{ref}

\appendix
\section{Auxiliary results}\label{sec:aux}
\begin{proposition}\label{prop:G2general} Under \Cref{Cond:0bilinear,Cond:2Tail}, if \Cref{Cond:1bound} holds for one orthonormal basis $(\phi_i)_{i\geq 1}$ of $\mathcal H$, then it holds for every orthonormal basis
\end{proposition}
\begin{proof}
We will first verify that it exists a constant $C>0$ such that for any $\delta \in (0,1)$ we have that
\begin{align}
   \limsup_{n\to \infty} \Prob(\|S_n\|_{\mathcal H\to \R}>\frac{1}{\delta})\leq C\delta. \label{eqn:random55j39h2}
\end{align}
For that, let $V_i$ be the span of $\phi_1,\ldots, \phi_i$, and write 
\begin{align}
   \limsup_{n\to \infty} \Prob(\|S_n\|_{\mathcal H\to \R}>\delta^{-1})&\leq 
   \lim_{i\to \infty}\limsup_{n\to \infty} \Prob( \|S_n \circ P_{V_i^{\intercal}}\|_{\mathcal H\to \R}^2 \geq (2\delta)^{-2})+\Prob(\|S_n \circ P_{V_i}\|_{\mathcal H\to \R}^2 \geq (2\delta)^{-2})\nonumber\\
   & = 0+  \lim_{i\to \infty}\limsup_{n\to \infty}\Prob(\|S_n \circ P_{V_i}\|_{\mathcal H\to \R}^2 \geq (2\delta)^{-2})\label{eqn:randomdijv19},
\end{align}
where the equality follows as \Cref{Cond:1bound} holds for the orthonormal basis $(\phi_i)_{i\geq 1}$. For the remaining term, we have $\|S_n \circ P_{V_i}\|_{\mathcal H\to \R}^2 = \sum_{j=1}^i S_n(\phi_i)^2$. By \Cref{Cond:0bilinear} we have that $(S_n(\phi_1),\ldots, S_n(\phi_i))$ converges to a random vector $(Z_1,\ldots, Z_i)$ of normal random variables with mean 0 and covariance matrix $\Sigma_{jk} = \sigma(\phi_j,\phi_k)$. Then, by the Markov inequality we have
\begin{align*}
    \limsup_{n\to \infty}\Prob\left(\|S_n \circ P_{V_i}\|_{\mathcal H\to \R}^2 \geq (2\delta)^{-2}\right)&= \Prob\left(\sum_{j=1}^i Z_i^2\geq (2\delta)^{-2}\right)
     \leq  4\delta^2\sum_{j=1}^i \sigma(\phi_j,\phi_j)\leq 4\delta\sum_{j=1}^{\infty} \sigma(\phi_j,\phi_j).
\end{align*}
Choose $C = 4\sum_{j=1}^{\infty}\sigma(\phi_j,\phi_j)$, then from \cref{eqn:randomdijv19} get $  \limsup_{n\to \infty} \Prob(\|S_n\|_{\mathcal H\to \R}>\delta^{-1})\leq \delta C$, and recall that $C<\infty$ due to \Cref{Cond:2Tail}.

We proceed to prove the main statement of the proposition. Consider any orthonormal basis $(\psi)_{i\geq 1}$ of $\mathcal H$ and fix $\varepsilon>0$. For any fixed $i\in \N$, let $V_i$ be the span of $\phi_1,\ldots, \phi_i$, and $U_i$ be the span of $\psi_1,\ldots, \psi_i$. Then, for any $\delta\in (0,1)$ we have that for any given $i\geq 1$, it exists $t(i)$ such that the following holds:
$$\text{for every $j\leq i$, it exists $\tilde \phi_j \in U_{t(i)}$ with
$\LH{\phi_j-\tilde \phi_j} \leq \delta$.}$$ Moreover, we can choose the vectors $\tilde \phi_1,\ldots, \tilde \phi_i$ to be orthonormal (we can just choose $t(i)$ large enough so we can approximate the first $i$ basis vectors $\phi_j$ with arbitrary precision). For sake of the argument, we choose $t(i)$ in such a way that $t(i)<t(i+1)$, so it is clear that $t(i)\to \infty$ as $i$ grows to infinity.

Now, denote by $\widetilde V_i$ the span of $\tilde \phi_1,\ldots, \tilde \phi_i$, then since $\widetilde V_i$ is a subspace of $U_{t(i)}$
we have
\begin{align}
    \|S_n \circ P_{U_{t(i)}^{\intercal}}\|_{\mathcal H \to \R} \leq  \|S_n \circ P_{\widetilde V_i^{\intercal}}\|_{\mathcal H \to \R} = \|S_n \circ P_{ V_i^{\intercal}}\|_{\mathcal H \to \R}+ \|S_n\|_{\mathcal H\to \R} \| P_{V_i}-P_{\widetilde V_i}\|_{\mathcal H \to \mathcal H}\label{eqn:randomvuru28416}
\end{align}
For the first term of the right-hand side we have that 
$\lim_{i\to \infty}\limsup_{n\to \infty}\Prob(\|S_n \circ P_{ V_i^{\intercal}}\|_{\mathcal H \to \R}\geq \varepsilon/2) = 0, $
since \Cref{Cond:1bound} holds for the orthonormal basis $(\phi_i)_{i\geq 1}$. For the second term on the right-hand side of \cref{eqn:randomvuru28416}, we first need to argue that $\| P_{V_i}-P_{\widetilde V_i}\|_{\mathcal H \to \mathcal H}\leq 2\delta$ . We start by noting that
\begin{align}
    \| (P_{V_i}-P_{\widetilde V_i})\|_{\mathcal H \to \mathcal H} = \| (I-P_{\widetilde V_i})P_{V_i})\|_{\mathcal H \to \mathcal H}+\| P_{\widetilde V_i}(I-P_{V_i}))\|_{\mathcal H \to \mathcal H}.
\end{align}
We will prove that both terms in the right-hand side of the equation above are smaller than $\delta$. We only do this for the first one since $\| P_{\widetilde V_i}(I-P_{V_i})\|_{\mathcal H \to \mathcal H} = \|(I-P_{V_i})P_{\widetilde V_i}\|_{\mathcal H \to \mathcal H}$, thus the same argument will work for both terms. Now, note that \begin{align}
      \| (I-P_{\widetilde V_i})P_{V_i})\|_{\mathcal H \to \mathcal H}^2 
      &= \sup_{\unitball}\sum_{j=1}^{i} \InerH{\omega}{\phi_j}^2\LH{(I-P_{\widetilde V_i})\phi_j}^2\nonumber 
\end{align}
but $\LH{(I-P_{\widetilde V_i})\phi_j}^2 $ is smaller than $\delta^2$ since $\tilde \phi_j$ in $\widetilde V_i$ is such that $\LH{\phi_j-\widetilde \phi_j}<\delta$. We deduce then that
$\| (I-P_{\widetilde V_i})P_{V_i})\|_{\mathcal H \to \mathcal H}^2 \leq \sup_{w\in V_i: \LH{\omega}=1}\sum_{j=1}^{i} \InerH{\omega}{\phi_j}^2\delta^2 \leq \delta^2,$ concluding that $\| P_{V_i}-P_{\widetilde V_i}\|_{\mathcal H \to \mathcal H}\leq 2\delta $ for any $i\geq 1$. The previous bound, together with \cref{eqn:random55j39h2}, yields
\begin{align*}
    \limsup_{n\to \infty}\Prob\left(\|S_n\|_{\mathcal H\to \R} \| P_{V_i}-P_{\widetilde V_i}\|_{\mathcal H \to \mathcal H}>\frac \varepsilon 2\right) \leq \limsup_{n\to \infty}\Prob\left(\|S_n\|_{\mathcal H\to \R}  >\frac{\varepsilon}{4\delta}\right)\leq \frac{4 C\delta}{\varepsilon}.
\end{align*}
Therefore $\lim_{i\to \infty}\limsup_{n\to \infty}  \Prob\left(\|S_n\|_{\mathcal H\to \R} \| P_{V_i}-P_{\widetilde V_i}\|_{\mathcal H \to \mathcal H}>\varepsilon/2\right) \leq \frac{4C\delta }{\varepsilon}$. From \cref{eqn:randomvuru28416} we deduce that 
$ \lim_{i\to \infty}\limsup_{n\to \infty}\Prob(\|S_n \circ P_{U_{t(i)}^{\intercal}}\|_{\mathcal H \to \R}\geq \varepsilon)\leq \frac{4\delta C}{\varepsilon}.$ To conclude, the previous inequality holds for the whole subsequence $(t(i): i\geq 1)$, instead of the whole sequence $(i:i\geq 1)$, however, as $\limsup_{n\to \infty}\Prob(\|S_n \circ P_{U_{i}^{\intercal}}\|_{\mathcal H \to \R}\geq \varepsilon)$ is decreasing in $i$, then limit exists, and thus $\lim_{i\to \infty} \limsup_{n\to \infty}\Prob(\|S_n \circ P_{U_{i}^{\intercal}}\|_{\mathcal H \to \R}\geq \varepsilon)\leq 4C\delta/\varepsilon$. Finally, since $\delta\in (0,1)$ is arbitrary the limit is  0.

\end{proof}

\begin{lemma}\label{lemma:lambda3ST}
Let $(\lambda_i)_{i\geq 1}$ be a sequence of non-negative real numbers and let $(Z_i)_{i\geq1}$ be a collection of i.i.d. standard normal random variables. Then $\sum_{i\geq1} \lambda_i < \infty$ if and only if $\sum_{i\geq1}\lambda_i Z_i^2$ converges almost surely to a random variable.
\end{lemma}

\begin{proof}

$(\Longrightarrow)$ We will prove that $\sum_{i=1}\lambda_i<\infty$ implies that the random series $\sum_{i=1}^n \lambda_iZ_i^2$ converges almost surely.  To do this, we verify the two conditions of the Kolmogorov's two series theorem. We first need to verify that $\sum_{i=1}^{\infty}\E(\lambda_iZ_i^2)<\infty$, but this follows immediately since $\E(Z_i^2) = 1$. We also need to verify that $\sum_{i=1}^{\infty}\var(\lambda_iZ_i^2)<\infty$, which follows immediately because $\var(Z_i^2) = 2$, and $\sum_{i=1}^{\infty}\lambda_i<\infty$ implies that $\sum_{i=1}^{\infty}\lambda_i^2<\infty$.

$(\Longleftarrow)$ We proceed to prove that if $\sum_{i=1}^\infty\lambda_iZ_i^2$ converges almost surely, then $\sum_{i=1}^\infty \lambda_i<\infty$. Note that since $\sum_{i=1}^\infty\lambda_iZ_i^2$ converges almost surely, the Kolmogorov's three series theorem deduces that for any $A>0$, it holds
\begin{align*}
i)\quad\sum_{i=1}^\infty\Prob(\lambda_iZ_i^2\geq A)<\infty, \qquad\text{and}\qquad
ii)\quad \sum_{i=1}^\infty\lambda_i\E\left(Z_i^2\Ind_{\{\lambda_i Z_i^2\leq A\}}\right)<\infty.
\end{align*}
We use i) to deduce that the sequence $(\lambda_i)_{i=1}^\infty$ is bounded. Consider $A=1$, and suppose, for contradiction, that there exists a sub-sequence $(\lambda_{n_k})_{k=1}^\infty$ such that $\lambda_{n_k}\to\infty$ as $k\to\infty$. Then, there exists $N\in\mathbb{N}$ large enough such that $\Prob(\lambda_{n_k}Z_{n_k}^2\geq 1)\geq 1/2$ for all $k\geq N$, and thus $\sum_{k=1}^\infty\Prob(\lambda_{n_k}Z_{n_k}^2\geq 1)\to\infty$ which contradicts i). We conclude that $\max_k\lambda_k$ is bounded by some constant.

Let $C=\E\left(Z_i^2\Ind_{\{(\max_{k}\lambda_k)Z_i^2\leq 1\}}\right)$ which is independent of $i$, then
\begin{align*}
   C\sum_{i=1}^n\lambda_i = \sum_{i=1}^\infty\lambda_i\E\left(Z_i^2\Ind_{\{(\max_{k}\lambda_k) Z_i^2\leq 1\}}\right)\leq  \sum_{i=1}^\infty\lambda_i\E\left(Z_i^2\Ind_{\{\lambda_i Z_i^2\leq 1\}}\right)< \infty.
\end{align*}\end{proof}

\section{Deferred proofs}

We prove \Cref{Lemma:approx} that by following the same steps as the proof of Theorem 6 of \cite{shah2020hardness}.

\begin{proof}[Proof of \Cref{Lemma:approx}]
Let $\omega\in \mathcal H$ with $\LH{\omega}=1$. We start by writing $S_n(\omega)$ as
\begin{align*}
    S_n(\omega)&=\frac{1}{\sqrt{n}}\sum_{i=1}^n \widehat \epsilon_{X_i}(Z_i)\widehat \epsilon_{Y_i}(Z_i)\omega(Z_i)
    =\widetilde{S}_n(\omega)+\nu_f+\nu_g+b,
\end{align*}
where
\begin{align*}
    \nu_g&=\frac{1}{\sqrt{n}}\sum_{i=1}^n\epsilon_{X_i}(Z_i)(g(Z_i)-\widehat{g}(Z_i))\omega(Z_i),\quad \nu_f=\frac{1}{\sqrt{n}}\sum_{i=1}^n\epsilon_{Y_i}(Z_i)(f(Z_i)-\widehat{f}(Z_i))\omega(Z_i),\\
    b&=\frac{1}{\sqrt{n}}\sum_{i=1}^n(f(Z_i)-\widehat{f}(Z_i))(g(Z_i)-\widehat{g}(Z_i))\omega(Z_i).
\end{align*}
The conclusion of the lemma follows from proving that  $\nu_g=o_p(1)$, $\nu_f=o_p(1)$ and $b=o_p(1)$, where the hidden constant in the $o_p$ notation is independent of $\omega$.

By \Cref{Cond:condIndep}.iv,  we have $\omega(z)^2=\InerH{\omega}{K_z}^2\leq C$ then the Cauchy-Schwarz's inequality yields
\begin{align*}
    b^2&\leq \frac{1}{n}{\sum_{i=1}^n(f(Z_i)-\widehat{f}(Z_i))^2}{\sum_{i=1}^n(g(Z_i)-\widehat{g}(Z_i))^2\omega(Z_i)^2}\leq {nCA_fA_g}=o_p(1),
\end{align*}
where last equality holds since under \Cref{Cond:condIndep}.i. 
We continue with the terms $\nu_g$ and $\nu_f$. Note that conditioned on $(\boldsymbol{Y,Z})=(Y_i,Z_i)_{i=1}^n$, $\nu_g$ is a sum of i.i.d. zero-mean random variables with variance:
\begin{align*}
    \Var(\epsilon_{X_i}(Z_i)(g(Z_i)-\widehat{g}(Z_i))\omega(Z_i)|\boldsymbol{Y,Z})&=\E(\epsilon_{X_i}(Z_i)^2(g(Z_i)-\widehat{g}(Z_i))^2\omega(Z_i)^2|\boldsymbol{Y,Z})\\
    &=(g(Z_i)-\widehat{g}(Z_i))^2\omega(Z_i)^2 u_P(Z_i,Y_i)\leq C'(g(Z_i)-\widehat{g}(Z_i))^2\omega(Z_i)^2,
\end{align*}
where, $C'>0$ is a constant such that $u_P(z,y)\leq C'$ uniformly for all pairs $(z,y)$, which exists by  \Cref{Cond:condIndep}.ii. Therefore,
\begin{align*}
    \E(\nu_g^2|\boldsymbol{Y,Z})&=\frac{1}{n}\sum_{i=1}^nC'(g(Z_i)-\widehat{g}(Z_i))^2\omega(Z_i)^2\leq\frac{C'C}{n}\sum_{i=1}^n(g(Z_i)-\widehat{g}(Z_i))^2=o_p(1).
\end{align*}

Finally, for any $\delta>0$, we have that
\begin{align*}
    \Prob(\nu_g^2\geq \delta)=\Prob(\nu_g^2\wedge \delta\geq \delta)\leq\delta^{-1}\E(\nu_g^2\wedge \delta)=\delta^{-1}\E(\E(\nu_g^2\wedge \delta|\boldsymbol{Y,Z}))\leq\delta^{-1}\E(\E(\nu_g^2|\boldsymbol{Y,Z})\wedge \delta)\to0,
\end{align*}
where the limit follows from Lebesgue's dominated convergence theorem since  $\E(\nu_g^2|\boldsymbol{Y,Z})\wedge \delta \overset{\Prob}{\to} 0$ as $n$ grows to infinite. From there we deduce that $\nu_g=o_p(1)$. By replicating the argument for $\nu_f$, we obtain that $\nu_f= o_p(1)$ as well. Note that our arguments do not use any bounds depending on $\omega$.
\end{proof}

\end{document}